\documentclass[a4paper,final]{article}
\usepackage{amsmath}
\usepackage{amsthm}
\usepackage{amssymb}
\usepackage{graphicx}
\usepackage{multirow}
\usepackage{longtable}
\usepackage{wrapfig}
\usepackage{float}
\usepackage{xcolor}
\usepackage[dvipsnames]{xcolor}
\usepackage{pdflscape}
\usepackage{hyperref}
\usepackage{setspace}
\usepackage{mathtools}
\usepackage{algorithm}
\usepackage{algpseudocode}
\usepackage{pgfplots}
\usepackage{adjustbox}
\renewcommand{\algorithmiccomment}[1]{\bgroup\hfill//~#1\egroup}
\usepackage{eurosym}
\usepackage[utf8]{inputenc}
\usepackage[english]{babel}
\usepackage{amsfonts}
\usepackage{authblk}
\usepackage{svg}
\svgpath{{figures/}}
\usepackage[normalem]{ulem}
\usepackage[backend=biber,style=numeric,sorting=nyt,maxnames=99]{biblatex}
\DeclareNameAlias{sortname}{family-given}
\DeclareNameAlias{default}{family-given}

\newtheorem{lemma}{Lemma}

\title{Resource-dependent process times in hybrid flexible flowshops}
\author[1,*]{Ioannis Avgerinos}
\author[1]{Ioannis Mourtos}
\author[1]{Dimitrios Papathanasiou}
\author[1,2]{Georgios Zois}
\affil[1]{\small ELTRUN Lab, Department of Management Science and Technology, Athens University of Economics and Business,\protect\\ \small Patision 76, 10434, Athens, Greece E-mail: {\tt \{iavgerinos, mourtos, dpapathanasiou, georzois\}@aueb.gr}}
\affil[2]{\small Optiscale, Athens 114 72, Greece}

\affil[*]{Corresponding Author} 
\date{}

\begin{document}
\maketitle

\begin{abstract}
The effect of resource allocation on manufacturing motivates us to examine a scheduling variant that is of practical significance yet remains overlooked. We examine a Hybrid Flexible Flowshop (HFFS), i.e., an environment where a set of jobs is scheduled across multiple stages (each stage having multiple identical machines) yet some jobs may skip some stages. In addition, we consider processing times that depend on the resources assigned to a job at each stage, transportation times between machines and limited-capacity buffers before and after each stage. We introduce a Constraint Programming (CP) formulation, which we then decompose through Logic-Based Benders Decomposition (LBBD). We tighten formulations by a set of makespan lower bounds, the strongest of which arises from a reduction to malleable scheduling. By modifying recent instance generators, we experiment with up to 400 jobs, 8 stages, and 10 parallel machines per stage. The results demonstrate competitive integrality gaps, highlighting the efficiency of our approach at scale and on an HFFS variant quite beyond the current literature.
\end{abstract}


\hfill \break
\textbf{Acknowledgments.} This work has been supported by Horizon Europe as part of the the MODAPTO Research and Innovation Action [grant number 101091996] \url{https://modapto.eu/}]

\section{Introduction} \label{sec:intro}

The growing complexity of real-world scheduling problems requests scalable solutions that handle realistic conditions, including the Industry 5.0 shift towards human-centric production layouts. Under this scope, benchmark problems are now shaped to reflect practical settings and often the variation of process times with the number of workers allocated.

Flowshop scheduling~\cite{johnson1954optimal} is the problem where a set of jobs is processed by a series of machines. Studies on flowshop scheduling have gradually evolved in both focus and methodology~\cite{gupta2006flowshop}. A notable extension is the Hybrid Flexible Flowshop problem (HFFS)~\cite{ruiz2008modeling}, where each production stage comprises multiple (and typically identical) parallel machines. The HFFS differs from other flowshop variants~\cite{ruiz2010hybrid}, as it reflects the ability of jobs to bypass certain stages. To mirror practical shop scheduling constraints, studies introduce HFFS formulations that incorporate additional conditions, such as sequence-dependent setup times~\cite{jabbarizadeh2009hybrid}, transportation times between stages or machines of the same~\cite{fernandez2025transportation} or different factories~\cite{perezgonzalez24}, and limited-capacity buffers~\cite{vatikiotis2024makespan}. Commonly, the problem is studied under the objective of minimising the makespan, i.e., the maximum completion time across all operations in the system.

In actual plant floors, however, processing times are resource-dependent~\cite{grigoriev2007machine}, meaning they vary based on the allocated resources and, most commonly, the assigned workforce. The utilised resources influence both the model constraints~\cite{behnamian2011hybrid} and the problem objective~\cite{mokhtari2011multi}. Despite the numerous variants addressed in the literature, the framework that incorporates both practical conditions (such as transportation times and limited-capacity buffers) and resource-dependent processing times has been largely overlooked, to the best of our knowledge. In fact, we have encountered such conditions in a gear-motor plant producing several hundreds of motors daily (see the acknowledgment). Therefore, motivated by both research gaps and business needs, our models and method account for machine-to-machine transportation times, as well as limited-capacity buffers and 
process times varying by both the production stage and the number of the workers assigned to a job.
    

\textbf{Literature review. } We review the literature in brief and the interested reader may find further elaboration within most papers discussed here. The standard formulation of Hybrid Flowshops (HFS)~\cite{linn1999hybrid} considers a series of productions stages, each containing multiple (commonly identical) machines. This setting has been mainly approached through the implementation of metaheuristics~\cite{choong2011metaheuristic}. Several algorithms have been proposed, from typical search-based models ~\cite{wang2009tabu, mirsanei2011simulated} to population-based metaheuristics~\cite{ruiz2006genetic, 
}. Such implementations are especially suitable to handle large instances of real-world cases.

With the inclusion of more case-realistic features, such as the flexibility factor, the selection of exact methods simplifies the way to express the additional constraints. A set of state-of-the-art Mixed-Integer Linear Programming (MILP) models is introduced by~\cite{naderi2014hybrid}, tackling datasets with up to 12 jobs and 4 stages. Under a different scope,~\cite{mollaei2019bi} propose a bi-objective MILP model for the HFFS with resource constraints, while another dual objective formulation, with makespan and tardiness criteria, is presented in~\cite{jungwattanakit2009comparison}. Extensions of typical MILP formulations of the HFFS are also established in~\cite{meng2020more} to incorporate real-world conditions. The testbeds used in the aforementioned studies prove that MILP models are not preferred for practical manufacturing conditions, which typically refer to large job instances. These cases require Constraint Programming (CP) formulations, which tackle real-world constraints more efficiently. The CP model presented in~\cite{icsik2023constraint}, tested in instances of up to 50 jobs and 4 stages, significantly outperforms the compared MILP in terms of both optimality gap and computational time. A CP implementation for the reentrant HFS is shown in~\cite{mlekusch2025dual}, which addresses both small (up to 50 jobs) and large (up to 100 jobs) testbeds. Highlighting the contemporary HFFS variant that accounts for transportation times between stages,~\cite{armstrong2021hybrid} introduce a state-of-the-art CP model which handles instances of up to 400 jobs and 8 production stages; this paper is particularly inspiring for our technical contribution.

Due to the scalability issues associated with exact methods, the use of decomposition approaches has recently gained traction. In a simple two-stage setting, ~\cite{tan2018logic} propose a Logic-Based Benders Decomposition (LBBD) model, distinguishing the implementation to an MILP master problem, and CP-based subproblems.~\cite{jiang2023decomposition} also adapt an MILP-CP model for a more complex HFFS environment. The inclusion of metaheuristics that decompose the established problem is shown in~\cite{wang2024decomposition}, where the proposed artificial bee colony algorithm is used to refine the obtained solutions through different modification phases. In a novel decomposition approach,~\cite{garraffa2025} present two types of heuristics that address different types of instances, for the HFFS with transportation times.

With research focus shifted towards settings that incorporate renewable resources \cite{jiang2023decomposition}, mainly referring to assigned workers \cite{gong2020energy}, it is important to adjust the HFFS frameworks according to the newly-surfaced constraints. A challenging feature comes from resource-based processing times~\cite{behnamian2011hybrid, 
}, meaning that the time needed to complete the job processing is influenced by workforce allocation in the manufacturing system.

\textbf{Contribution. }
 Our work is motivated by a large-scale and elaborate HFFS environment (i.e., with buffers, transportation times, limited resources and resource-dependent process times) and by the lack of methods that exploit the inherently decomposed nature of the problem under renewable constraints. We contribute in both directions by models, methods and computational evidence on problems of realistic size.

MILP models become inefficient as flowshop scheduling problems scale up, even with a relatively small number of jobs. In contrast, CP models effectively handle precedence and resource constraints. However, their inability to compute strong dual bounds limits the ability to assess the quality of solutions. To address large-scale datasets, we propose a LBBD algorithm composed entirely of CP models for both the master problem and the subproblem, integrated using logic Benders cuts. To mitigate CP’s limitations in producing strong lower bounds, we develop eight different dual bounds. These bounds not only tighten the search within the solution space but also significantly enhance the evaluation of solution quality by reducing the optimality gap. The most impactful bound is plausible after reducing the problem to a malleable scheduling relaxation that supports a broad class of practical, resource-dependent processing time functions.

\textbf{Outline. }Section \ref{sec:prel} elaborates on the HFFS with resource-dependent processing times, providing a description of the problem and the mathematical notation adopted throughout the paper. It also outlines the theoretical framework of LBBD and our decomposition approach. Section \ref{sec:model} presents the formulations of the CP model and the LBBD partitions, together with the employed lower bounds and the malleable scheduling reduction of the problem.
Section \ref{sec:experiments} describes the experimental work and introduces key metrics used to assess the performance of the methods.
The paper concludes with the remarks in Section \ref{sec:conclusions} and an Appendix containing the detailed experimental results.
\section{Preliminaries} \label{sec:prel}
    \subsection{Problem description}
    A set of jobs $J$ must be processed through a sequence of production stages $S$, each equipped with a set of identical machines $M$. Each machine is assigned to a specific stage $s$, where the subset $M_{s}\subset M$ denotes the machines available at stage $s$. Each job $j\in J$ may skip certain stages; hence, the subset $S_{j}\subseteq S$ represents the sequence of stages that are eligible for processing job $j$.
        
        Jobs are temporarily stored in limited-capacity buffers at the entry ($R^{-}_{m}$) and exit ($R^{+}_{m})$ of each machine $m\in M$, when they are not undergoing processing. Additionally, we account for transportation times between machines of subsequent stages, where $t_{mn}$ represents the time needed to transfer a job from machine $m$ of stage $s$ to machine $n$ of the next eligible stage $s^\prime$. To indicate the order of stages for each job, we use the indexed annotation $s^{i}_{j}$, denoting the $i^{\text{th}}$ stage of job $j\in J$, $i\in \{1, ..., |S_{j}|\}$. The environment operates with a number of workers $W$. For each stage $s$, all eligible jobs can be processed by a minimum ($w^{-}_{s}$) and a maximum ($w^{+}_{s}$) number of assigned workers. Each job is associated with a process time $p_{jsw}$, which is determined by both the stage $s$ and the number of occupied workers $w$, through a function $f_{js}(Q)$, $Q$ being a subset of workers, implying that $w = |Q|$. For brevity, we can replace the operations (i.e., the pairs of jobs and services $j, s$) with $j'$. Table \ref{tab:notation} summarises the mathematical annotation of the problem.

        \begin{table}[H]
            \centering
            \caption{Mathematical annotation}\label{tab:notation}
            \begin{tabular}{|l l|}
            \hline
            Sets &  \\
            \hline
            $J$ &   Set of jobs \\
            $S$ &   Set of stages \\
            $M$ &   Set of machines \\
            $S_{j}\subseteq S$  & Set of stages which are eligible for job $j\in J$ \\
            $M_{s}\subset M$ & Set of machines dedicated to stage $s\in S$ \\
            \hline
            Parameters &   \\
            \hline
            $R^{-}_{m}$ &  Buffer capacity on the entry of machine $m\in M$ \\
            $R^{+}_{m}$ &  Buffer capacity on the exit of machine $m\in M$ \\
            $t_{mn}$    & Transportation time between machines $m$ and $n$ \\
            $W$ & Total number of workers \\
            $w^-_{s}$ & Minimum number of workers required to process any job at stage $s\in S$ \\
            $w^+_{s}$ & Maximum number of workers allowed to process any job at stage $s\in S$ \\
            $p_{jsw}$ &  Processing time of job $j\in J$ on stage $s\in S_{j}$ if $w\in [w^{-}_{s}, w^{+}_{s}]$ workers are allocated \\
            $f_{js}(Q)$ &  Processing time function of job $j$ on stage $s$ if a subset of workers $Q$ is allocated \\
            $s^{i}_{j}\in S_{j}$ & The $i^{\text{th}}$ stage of job $j\in J$, $i\in \{1, ..., |S_{j}|\}$ \\
            \hline
        \end{tabular}
        \end{table}

        To illustrate the notation, consider the example shown in Figure \ref{fig:exHFFS}. We present an environment with three production stages, meaning $S = \{s_1, s_2, s_3\}$. Each stage comprises a varying number of machines. The aim is to schedule the defined set of jobs $J = \{j_1, j_2, j_3\}$. The edges between machines represent the transportation times required to transfer a job from one machine to another. Jobs can remain at entry (`In') or exit ('Out') buffers, while the machines are occupied by other jobs. Some jobs may be allowed to be transferred from $s_1$ to $s_3$, but we opt not to depict the respective edges for visualisation purposes.

        \begin{figure}[h]
            \centering
                \includegraphics{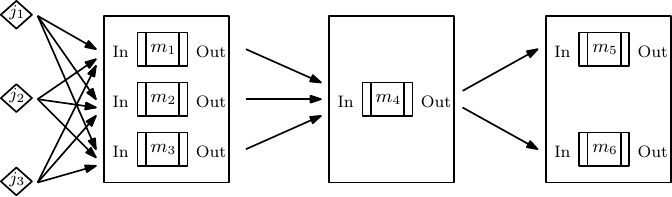}
                \caption{An indicative example of three jobs on three stages}
                \label{fig:exHFFS}
        \end{figure}

    \subsection{Logic-Based Benders Decomposition}
        Classical Benders Decomposition \cite{benders62} is one of the most widely applied partitioning methods for MILPs. It relies on duality-based optimality and feasibility cuts to decompose the original problem into a master problem (an Integer Linear Program) and a subproblem, which is a Linear Program consisting solely of continuous variables. This restriction to purely continuous subproblems, however, limits the applicability of the method in many optimisation problems where the subproblem naturally involves integer variables or requires alternative mathematical modelling frameworks.

        Hooker \cite{hooker03} extended the classical Benders decomposition into LBBD, a more flexible partitioning approach that guarantees convergence to optimality, regardless of the structure of the master and subproblem formulations. Given the efficiency of CP in handling renewable resources, thanks to the \texttt{Cumulative} predicate, \cite{hooker07} proposed LBBD variants for scheduling problems in which resource allocation is managed by a CP-formulated subproblem. A summary of the method follows:
    
        Let $\mathcal{P}$ be a minimisation problem of two groups of variables $x$, $y$ in domains $D_{x}$ and $D_{y}$: 
        \begin{equation}
            \mathcal{P}:\qquad min\{a(x) + b(y):C_{1}(x), C_{2}(y), C_{3}(x, y), x\in D_{x}, y\in D_{y}\}. \notag 
        \end{equation}
        The problem is subject to constraints $C_{1}(x)$, $C_{2}(y)$ and $C_{3}(x, y)$. 
        To partition $\mathcal{P},$ $\mathcal{M}$ is defined as
        \begin{equation}
            \mathcal{M}:\qquad min\{z:C_{1}(x), x\in D_{x}, z\geq a(x)\}. \notag 
        \end{equation}
                
        Assuming that $a(x)$ and $b(y)$ are non-negative cost functions, $z$ is a lower bound of the objective function of $\mathcal{P}.$ All constraints of $\mathcal{M}$ are also imposed to $\mathcal{P}$. Therefore, $\mathcal{M}$ is a \emph{relaxation} of $\mathcal{P}$, called the \emph{master problem}.
        
        Given the optimal solution $\bar{x}$ of $\mathcal{M}$ and the respective objective value $\bar{z}$, the following problem $\mathcal{S}$ computes a local optimum of $\mathcal{P}$:
        \begin{equation}
            \mathcal{S}:\qquad min\{a(\bar{x}) + b(y):C_{2}(y), C_{3}(\bar{x}, y), y\in D_{y}\}. \notag 
        \end{equation}
        Notably, $\mathcal{S}$ is a \emph{subproblem} of $\mathcal{P}$ that finds the optimal solution in the search space restricted by $\bar{x}$. If $C_{3}(\bar{x}, y)$ are satisfied for the given solution $\bar{x}$, the optimal objective value of $\mathcal{S},$ denoted as $\bar{\zeta},$ is an upper bound of the optimal objective value of $\mathcal{P}$. Hence, $\frac{\bar{\zeta} - \bar{z}}{\bar{\zeta}}$ defines a valid optimality gap.
        
        That is, LBBD employs an iterative exchange of knowledge between $\mathcal{M}$ and $\mathcal{S}$: $\mathcal{M}$ provides a lower bound of $\mathcal{P}$ and a partial solution $\bar{x}$, and, then, $\mathcal{S}$ returns an upper bound which is the local optimum of the restricted search space. All subsequent iterations must prevent the master problem from providing previous partial solutions, unless these are proved to be the globally optimal ones. This can be achieved by the generation of a valid \emph{optimality cut}, if the partial solution leads to a feasible solution of $\mathcal{S}$:
        \begin{equation}
            \text{if }x = \bar{x} \rightarrow z\geq \bar{\zeta}. \notag
        \end{equation}
        On the contrary, if $\mathcal{S}$ is infeasible for a partial solution $\bar{x}$, then a logic constraint called \emph{feasibility cut} is generated:
        \begin{equation}
            x \neq \bar{x}.\notag
        \end{equation}
        Traditionally, Benders cuts are linearised to fit within a MILP-formulated master problem. In contrast, this paper proposes an alternative approach in which the master problem is modeled using CP, thus the Benders cuts are generated in a logic form.
\section{Mathematical Modelling} \label{sec:model}
    \subsection{The CP model}
In recent decades, CP has become a leading approach for modelling complex scheduling problems. While MILP formulations are generally more effective at producing tight lower bounds, they suffer from poor scalability, limiting their practicality in real-world production environments. Position-based formulations \cite{ruiz14} have achieved low optimality gaps in flowshop scheduling problems, sometimes as part of a Benders decomposition scheme \cite{ruiz20}, by incorporating precedence constraints between stages without relying on big-M constraints, which are known to weaken bounds. However, additional complexities common in industrial settings, such as machine-to-machine transportation times, are difficult to model without compromising bound tightness. Similarly, incorporating renewable resource constraints (e.g., limited workforce availability) typically requires time-indexed variables, which lead to prohibitively large models, even for small problem instances.
        
By contrast, CP offers the necessary modelling flexibility, particularly through the use of interval variables and global cumulative constraints. This flexibility is especially valuable for problems that require optional interval variables; nonetheless, this comes at a significant cost in terms of lower bound quality. Recent advances in solver technology and modelling strategies have improved CP’s performance on standard scheduling benchmarks, including flowshop variants.

Given these considerations, the problem addressed in this work is formulated using a CP model that integrates all relevant constraints (i.e., limited workforce, capacitated buffers, precedence constraints) and the resource-dependent aspect of processing times:

    \begingroup
        \footnotesize
        \begin{flalign}
            \mathcal{CP}: & &&\notag &&\\
            \text{min } & C_{max} &&\notag &&\\
                        & C_{max}\geq \texttt{endOf}(\texttt{jobInterval}_{js}) &&\forall j\in J, s\in S_{j} \label{eq:c1} &&\\
                        & \texttt{alternative}(\texttt{jobInterval}_{js}, [\texttt{process}_{jm}|m\in M_{s}]) && \forall j\in J, s\in S \label{eq:c2} &&\\
                        &\text{if }\texttt{machineOf}_{js}=m\rightarrow \texttt{presenceOf}(\texttt{waitBefore}_{jm}) = 1 &&\forall j\in J, s\in S_{j}, m\in M_{s} \label{eq:c3} &&\\
                        &\text{if }\texttt{machineOf}_{js}=m\rightarrow \texttt{presenceOf}(\texttt{process}_{jm}) = 1 &&\forall j\in J, s\in S_{j}, m\in M_{s} \label{eq:c4} &&\\
                        &\text{if }\texttt{machineOf}_{js}=m\rightarrow \texttt{presenceOf}(\texttt{waitAfter}_{jm}) = 1 &&\forall j\in J, s\in S_{j}, m\in M_{s} \label{eq:c5} &&\\
                        & \texttt{startAtEnd}(\texttt{process}_{jm}, \texttt{waitBefore}_{jm}) &&\forall j\in J, m\in M \label{eq:c6} &&\\
                        & \texttt{startAtEnd}(\texttt{waitAfter}_{jm}, \texttt{process}_{jm}) && \forall j\in J, m\in M \label{eq:c7} &&\\
                        & \texttt{alternative}(\texttt{process}_{jm}, [\texttt{allocateWorkers}_{jmw}|w\in \{w^{-}_{s}, ..., w^{+}_{s}\}]) &&\forall j\in J, m\in M \label{eq:c9} &&\\
                        &\texttt{noOverlap}([\texttt{process}_{jm}|j\in J]) &&\forall m\in M \label{eq:c11} &&\\
                        &\sum_{j\in J}\texttt{pulse}(\texttt{waitBefore}_{jm}, 1)\leq R^{-}_{m} &&\forall m\in M \label{eq:c12} &&\\
                        &\sum_{j\in J}\texttt{pulse}(\texttt{waitAfter}_{jm}, 1)\leq R^{+}_{m} &&\forall m\in M \label{eq:c13} &&\\
                        &\sum_{j\in J}\sum_{m\in M}\sum_{w = w^{-}_{s}}^{w^{+}_{s}]}\texttt{pulse}(\texttt{allocateWorkers}_{jmw}, w) \leq W &&\label{eq:c14} &&\\
                        &\texttt{startAtEnd}(\texttt{waitBefore}_{jn}, \texttt{waitAfter}_{jm}, -t_{mn}) &&\forall j\in J, i\in \{2, ..., |S_{j}|\}, m\in M_{s^{i-1}_{j}}, n\in M_{s^{i}_{j}} \label{eq:c15} &&\\
                        &\texttt{endBeforeStart}(\texttt{jobInterval}_{js^{i-1}_{j}}, \texttt{jobInterval}_{js^{i}_{j}}) &&\forall j\in J, i\in \{2, ..., |S_{j}|\} \label{eq:c16} &&\\
                        & &&\notag &&\\
            &\texttt{jobInterval}_{js}: \texttt{interval variables} && \forall j\in J, s\in S_{j} \notag &&\\
            &\texttt{waitBefore}_{jm}: \texttt{interval variables} && \notag &&\\
            &\qquad \qquad \texttt{optional} = \texttt{True} &&\forall j\in J, m\in M \notag &&\\
            &\texttt{process}_{jm}: \texttt{interval variables} && \notag &&\\
            &\qquad \qquad \texttt{optional} = \texttt{True} &&\forall j\in J, m\in M \notag &&\\
            &\texttt{waitAfter}_{jm}: \texttt{interval variables} && \notag &&\\
            &\qquad \qquad \texttt{optional} = \texttt{True} &&\forall j\in J, m\in M \notag &&
            \end{flalign}
            \begin{flalign}
            &\texttt{allocateWorkers}_{jmw}: \texttt{interval variables} &&\notag &&\\
                            &\qquad \qquad \texttt{sizeOf}(\texttt{allocateWorkers}_{jmw}) = p_{jsw}&&\notag &&\\
                            &\qquad \qquad \texttt{optional} = \texttt{True} &&\forall j\in J, s\in S_{j}, m\in M_{s} \notag &&\\
            &\texttt{machineOf}_{js}\in M_{s} &&\forall j\in J, s\in S_{j} \notag &&\\
            &C_{max}\geq 0 && \notag &&
        \end{flalign}
    \endgroup

    The CP model incorporates multiple interval variables to handle various resource constraints. The variables $\texttt{jobInterval}_{js}$ are non-optional, as each job $j$ must be processed at all stages $s\in S_j$. To determine the machine assignment at each stage $s$, we define optional interval variables: $\texttt{waitBefore}_{jm}$, $\texttt{process}_{jm}$, and $\texttt{waitAfter}_{jm}$, representing the waiting time before processing, the processing itself, and the waiting time after processing on machine $m$, respectively. To model resource-dependent processing times, additional optional variables $\texttt{allocateWorkers}_{jmw}$ represent the processing interval of job $j$ on machine $m$ with $w$ allocated workers. The duration of $\texttt{allocateWorkers}_{jmw}$ is set to $p_{jsw}$, where $s$ is the stage corresponding to machine $m$.

    The objective is to minimise the makespan $C_{max}$, which must be greater than the completion time of any operation $(j, s)$ (see Constraint \ref{eq:c1}). Each interval $\texttt{jobInterval}_{js}$ is synchronised with exactly one $\texttt{process}_{jm}$ for some machine $m \in M_s$ using the $\texttt{alternative}$ predicate (Constraint \ref{eq:c2}). The variable $\texttt{machineOf}_{js}$ indicates the selected machine for stage $s$. If $\texttt{machineOf}_{js} = m$, then the corresponding intervals $\texttt{waitBefore}_{jm}$, $\texttt{process}_{jm}$, and $\texttt{waitAfter}_{jm}$ are present (Constraints \ref{eq:c3}–\ref{eq:c5}). The execution order among $\texttt{waitBefore}_{jm}$, $\texttt{process}_{jm}$, and $\texttt{waitAfter}_{jm}$ is enforced by Constraints \ref{eq:c6}–\ref{eq:c7}: processing starts after waiting at the entry, and finishes before waiting at the exit.

    According to (\ref{eq:c9}), each $\texttt{process}_{jm}$ interval is synchronised with exactly one $\texttt{allocateWorkers}_{jmw}$, ensuring the correct duration based on the number of workers. The $\texttt{noOverlap}$ predicate (constraint \ref{eq:c11}) guarantees that only one job is processed on a machine at any time. Pulse constraints (\ref{eq:c12}) and (\ref{eq:c13}) restrict the number of jobs waiting at machine entry and exit to the respective buffer capacities. Similarly, Constraint (\ref{eq:c14}) limits the total number of workers simultaneously in use to the maximum available $W$. To account for transportation between machines of consecutive stages, Constraint (\ref{eq:c15}) enforces that the time between the end of $\texttt{waitAfter}_{jm}$ and the start of $\texttt{waitBefore}_{jn}$ equals the transportation time $t_{mn}$. Finally, the overall stage order is imposed by Constraint (\ref{eq:c16}).
    
    As demonstrated in Section \ref{sec:experiments}, the model can generate feasible solutions for instances involving hundreds of jobs. However, scaling up significantly affects performance, since the quality of the solutions drops sharply as the number of jobs increases.

    \subsection{Relaxation of the problem} \label{sec:master}
\textbf{HFFS with transportation times. }To address the complexity of the full problem, we partition the problem into a relaxed master problem, where processing times do not vary with the number of workers and buffer capacities are unlimited, and a subproblem, which reintroduces the relaxed constraints on workers and buffers. Given a fixed machine sequence for each job from the relaxation (i.e., the master problem), the subproblem adjusts the schedule to yield the optimal among all solutions of the original problem with the given machine sequence per job.

To construct a relaxation of the original problem, the resource-dependent processing times are replaced by resource-independent lower bounds. Specifically, for each job $j$ on stage $s$, we define: $\bar{p}_{js} = \text{min}\{p_{jsw}|w\in \{w^{-}_{s}, ..., w^{+}_{s}\}\}$. Additionally, each job is assumed to occupy a fixed number of workers equal to the minimum required $w^{-}_{s}$ for its stage; this ensures that both the processing times and resource usage in the relaxation are lower bounds of their actual values. Also, as mentioned before, the constraints regarding the capacity of buffers at entry and exit points of the machines are neglected. Consequently, the above substitutions yield a valid relaxation $\mathcal{M}$.
        
Thankfully, we can use a variation of the CP model of \cite{garraffa2025} for the HFFS with transportation times, which appears as currently the best-performing exact approach.
\begingroup
        \small
        \begin{flalign}
            \mathcal{M}:    & &&\notag &&\\
            \text{min }     &C_{max}    &&\notag &&\\
                            &C_{max}\geq \texttt{endOf}(\texttt{jobIntervalA}_{js}) &&\forall j\in J, s\in S_{j} \label{eq:m1} &&\\
                            &\texttt{alternative}(\texttt{jobIntervalA}_{js}, [\texttt{jobIntervalB}_{jm}|m\in M_{s}])       &&\forall j\in J, s\in S_{j} \label{eq:m2} &&\\
                            &\text{if }\texttt{machineOf}_{js} = m\rightarrow \texttt{presenceOf}(\texttt{jobIntervalB}_{jm}) = 1   &&\forall j\in J, s\in S_{j}, m\in M_{s} \label{eq:m3} &&\\
                            &\texttt{endBeforeStart}(\texttt{jobIntervalB}_{jm}, \texttt{jobIntervalB}_{jn}, t_{mn}) &&\forall j\in J, i\in \{2, ..., |S_{j}|\}, m\in M_{s^{i-1}_{j}}, n\in M_{s^{i}_{j}} \label{eq:m4} &&\\
                            &\texttt{endBeforeStart}(\texttt{jobIntervalA}_{js^{i-1}_{j}}, \texttt{jobIntervalA}_{js^{i}_{j}}) &&\forall j\in J, i\in \{2, ..., |S_{j}|\} \label{eq:m5} &&\\
                            &\texttt{noOverlap}([\texttt{jobIntervalB}_{jm}|j\in J])    &&\forall m\in M \label{eq:m6}&&\\
                            &\sum_{j\in J: s\in S_{j}}\texttt{pulse}(\texttt{jobIntervalA}_{js}, 1)\leq |M_{s}| &&\forall s\in S \label{eq:m7}&&\\
                            &\sum_{j\in J: s\in S_{j}}\texttt{pulse}(\texttt{jobIntervalA}_{js}, w^{-}_{s})\leq W && \label{eq:m8}&&\\
                            &\texttt{jobIntervalA}_{js}: \texttt{interval variables} &&\notag &&\\
                            &\qquad \qquad \texttt{sizeOf}(\texttt{jobIntervalA}_{js}) = \bar{p}_{js} &&\forall j\in J, s\in S_{j} \notag &&\\
                            &\texttt{jobIntervalB}_{jm}: \texttt{interval variables} &&\notag &&\\
                            &\qquad \qquad \texttt{optional} = \texttt{True} &&\forall j\in J, s\in S_{j}, m\in M_{s} \notag &&\\
                            &\texttt{machineOf}_{js}\in M_{s}   &&\forall j\in J, s\in S_{j}\notag &&
        \end{flalign}
        \endgroup

Model $\mathcal{M}$ uses a set of non-optional interval variables $\texttt{jobIntervalA}_{js}$, indicating the start and end times of job $j$ on stage $s$, given that the duration of the time interval is set to $\bar{p}_{js}$, i.e., the minimum processing time over all assigned workers. Optional time intervals $\texttt{jobIntervalB}_{jm}$ denote the start and end times of job $j$ on machine $m$. Variables $\texttt{machineOf}_{js}$ indicate the machine $m\in M_{s}$ on which stage $s$ of job $j$ is processed.
        
The objective function minimises the makespan, i.e., the maximum end time of all jobs (\ref{eq:m1}). Constraints (\ref{eq:m2}) use the predicate \texttt{alternative}, imposing that, for each stage $s$, at least one of variables $\texttt{jobIntervalB}_{jm}$ of machines $M_{s}$ is present. Variable $\texttt{jobIntervalA}_{js}$ is synchronised with the present variable $\texttt{jobIntervalB}_{jm}$ (i.e., they have the same start and end times). Conditional constraints (\ref{eq:m3}) indicate that the interval variable $\texttt{jobIntervalB}_{jm}$ of the assigned machine $m$ for each stage $s$ of job $j$ is present. Constraints (\ref{eq:m4}) indicate that job $j$ can start being processed at machine $n$, after the corresponding transportation time from machine $m$ of the precedent stage. The appropriate order of stages is imposed by Constraints (\ref{eq:m5}). Constraints (\ref{eq:m6}) use the \texttt{noOverlap} predicate to ensure that each machine can process no more than one job at the same time. The \texttt{pulse} predicate imposes a limited consumption of renewable resources at any time. Constraints (\ref{eq:m7}) are redundant, but evidently efficient: they ensure that the number of jobs that can be processed at the same stage simultaneously do not exceed the number of parallel machines of the stage. Last, Constraints (\ref{eq:m8}) ensure that no more than $W$ workers can be occupied at the same time, assuming that each job consumes the minimum number of workers $w^{-}_{s}$ at stage $s$.\\

\noindent \textbf{Lower bounds. }The use of optional interval variables in CP models is often associated with weaker lower bounds compared to integer programming formulations, as noted in \cite{naderi2023mixed}. To address this, we exploit structural properties of the problem to derive a set of rules inspired by classical makespan lower bounds, which are then used to impose a tighter bound on the variable $C_{max}$ in formulation $\mathcal{M}$.

A well-known lower bound on the makespan for scheduling on parallel identical machines is obtained by evenly distributing the total processing time across all machines: $C_{max}\cdot |M|\geq \sum_{j\in J}p_{j}$ where $p_{j}$ denotes the processing time of job $j$ across all machines, and $|M|$ is the number of parallel machines. This bound can be adapted within an HFFS framework by applying it to each stage individually.
        \begin{flalign}
            &C_{max}\geq \frac{\sum_{j\in J:s\in S_{j}}\bar{p}_{js}}{|M_{s}|} &&\forall s\in S \notag && (\text{LB}_{1})&& && && &&
        \end{flalign}
        To incorporate the flexibility of jobs in ($\text{LB}_{1}$), the total processing time at each stage $s$ includes only the jobs applicable to that stage. Furthermore, actual processing times are replaced by their relaxed counterparts $\bar{p}_{js}$, defined as the minimum processing time over all worker assignments for job $j$ at stage $s$.

        Another straightforward, yet often tighter, lower bound is based on the minimum possible completion time for each individual job. For a given job $j$, this minimum is computed as the sum of the relaxed processing times $\bar{p}_{js}$ over all eligible stages $s\in S_{j}$, plus the minimum transportation time $\bar{t}_{j}$ along the shortest feasible path through the flowshop environment. This path begins at a machine in the first eligible stage $s^{1}_{j}$, ends at a machine in the last stage $s^{|S_{j}|}_{j}$, and visits exactly one machine at each intermediate stage. The transportation time $\bar{t}_{j}$ can be efficiently obtained by solving a minimum-cost flow problem, formulated as the following Linear Program. 
        \begin{flalign}
            \text{(SP}_{j}\text{): } & &&\notag &&\\
            \text{min }& \bar{t}_{j} = \sum_{i=2}^{|S_{j}|}\sum_{m\in M_{s^{i-1}_{j}}}\sum_{n\in M_{s^{i}_{j}}}t_{mn}\cdot x_{mn} &&\notag &&\\
            &\sum_{m\in M_{s^{i-1}_{j}}}\sum_{n\in M_{s^{i}_{j}}}x_{mn} = 1 &&\forall i\in \{2, ..., |S_{j}|\} &&\notag &&\\
            &x_{mn}\geq 0 &&\forall i\in \{2, ..., |S_{j}|\}, m\in M_{s^{i-1}_{j}}, n\in M_{s^{|S_{j}|_{j}}} &&\notag &&
        \end{flalign}
        And subsequently:
        \begin{flalign}
            &C_{max}\geq \sum_{i=1}^{|S_{j}|}\bar{p}_{js^{i}_{j}} + \bar{t}_{j} &&\forall j\in J \notag && (\text{LB}_{2}) && && && &&
        \end{flalign}
        Combining ($\text{LB}_{1}$) and ($\text{LB}_{2}$) can yield a stronger lower bound on the makespan. Let $t^{s}_{j}$ denote the duration of the shortest path for job $j$ up to stage $s$, assuming $s\in S_{j}$. Similarly, let $p^{s-}_{j}$ represent the sum of relaxed processing times $\bar{p}_{js^{\prime}}$ for all eligible stages $s^{\prime}$ preceding $s$. It follows that stage $s$ cannot begin earlier than the minimum value of $p^{s-}_{j} + t^{s}_{j}$ over all jobs $j$ for which $s$ is an eligible stage. Moreover, by ($\text{LB}_{1}$), the minimum duration of stage $s$ is given by $\frac{\sum_{j\in J:s\in S_{j}}\bar{p}_{js}}{|M_{s}|}$. Therefore, the sum of the minimum start time and the minimum duration of stage $s$ provides a valid lower bound as follows.
        \begin{flalign}
            &C_{max}\geq \frac{\sum_{j\in J:s\in S_{j}}\bar{p}_{js}}{|M_{s}|} + \text{min}\Bigl\{p^{s-}_{j} + t^{s}_{j}|j\in J \text{ if } s\in S_{j}\Bigr\} &&\forall s\in S \notag && (\text{LB}_{3}) && && && &&
        \end{flalign}

Lower bounds for the minimisation of makespan on hybrid flowshop environments are presented by \cite{lee1994minimizing}. The validity of these bounds can be retained, as long as the flexibility of skipping stages is incorporated. We therefore consider partitions of $\mathcal{M}$ which consider only a pair of consecutive stages. Given a stage $s\in S$ and its consecutive stage $s+1$, let $j^{1}, j^{2}, ..., j^{n}$ be the jobs for which $s\in S_{j}$ and $s+1\in S_{j}$, in non-decreasing order of processing times on stage $s$.

Further inequalities are obtainable by reformulating the bounds of \cite{lee1994minimizing}, in which the interested reader may find a detailed presentation. Let for convenience $s^{\prime}=s+1$. If $|M_{s}|\geq |M_{s^{\prime}}|$, we obtain the following bounds on the optimal value of $\mathcal{M}$:
 \begin{flalign}
            &C_{max}\geq \frac{p_{j^{|M_{s^{\prime}}|}s} + p_{j^{n}s^{\prime}}}{|M_{s}|} && s\in S, s^{\prime} = s+1 \notag && (\text{LB}_{4}) && && && &&\\
            &C_{max}\geq \frac{p_{j^{|M_{s}|}s} + (|M_{s}| - |M_{s^{\prime}}|)\cdot p_{j^{1}s^{\prime}} + p_{j^{n}s}}{|M_{s}|} && s\in S, s^{\prime} = s+1 \notag && (\text{LB}_{5}) && && && &&
        \end{flalign}
        Reversely, if $|M_{s}|< |M_{s^{\prime}}|$:
        \begin{flalign}
            &C_{max}\geq \frac{p_{j^{|M_{s^{\prime}}|s}} + (|M_{s^{\prime}}|-|M_{s}|)\cdot p_{j^{1}s} + p_{j^{n}s^{\prime}}}{|M_{s^{\prime}}|} && s\in S, s^{\prime} = s+1 \notag && (\text{LB}_{6}) && && && &&\\
            &C_{max}\geq \frac{p_{j^{|M_{s}|}s^{\prime}} + p_{j^{n}s}}{|M_{s}|} && s\in S, s^{\prime} = s+1 \notag && (\text{LB}_{7}) && && && &&
        \end{flalign}

    \subsection{Subproblem}
The values of the variables $\texttt{machineOf}_{js}$ define the machines that process job $j$ at each stage $s$. We can now consider a set containing all these values, i.e., a set $J^*=\{j^*, j \in J\}$. For each job $j$, $j^{*}$ is a tuple representing the sequence of machines which are assigned to process $j$, i.e., $j^{*} = (m^{1}_{j}, m^{2}_{j}, ..., m^{|S_{j}|}_{j})$, where the superscript refers to the stage in the set of eligible stages $S_{j}$ (i.e., $m^{i}_{j} = \texttt{machineOf}_{js^{i}_{j}}$). With a slight abuse of notation we write $j\in J, m\in j^{*}$, although $j^*$ is an ordered subset of the set of machines $M$, to say that for a job $j$, machine $m$ shows up in the sequence of machines processing that job. 

The subsequent subproblem is responsible for assigning workers to jobs, taking into account the job-to-machine assignments defined by $J^*$, as well as the capacity constraints of the buffers at the entry and exit points of each machine:

        \begingroup
        \footnotesize
        \begin{flalign}
            \mathcal{S}:    & &&\notag &&\\
            \text{min }     &C_{max} &&\notag &&\\
            &C_{max}\geq \texttt{endOf}(\texttt{waitAfter}_{jm})&& \forall j\in J, m\in j^{*} \label{eq:s1} &&\\
            &\texttt{startAtEnd}(\texttt{process}_{jm}, \texttt{waitBefore}_{jm}) && \forall j\in J, m\in j^{*} \label{eq:s2} &&\\
            &\texttt{startAtEnd}(\texttt{waitAfter}_{jm}, \texttt{process}_{jm}) && \forall j\in J, m\in j^{*} \label{eq:s3} &&\\
            &\texttt{startAtEnd}(\texttt{waitBefore}_{jm^{i}_{j}}, \texttt{waitAfter}_{jm^{i-1}_{j}}, -t_{m^{i-1}_{j}m^{i}_{j}}) &&\forall j\in J, i\in \{2, ..., |S_{j}|\} \label{eq:s4} &&\\
            &\texttt{alternative}(\texttt{process}_{jm^{i}_{j}}, [\texttt{allocateWorkers}_{jm^{i}_{j}w}|w\in \{w^{-}_{s^{i}_{j}}, ..., w^{+}_{s^{i}_{j}}\}]) &&\forall j\in J, i\in \{1, ..., |S_{j}|\} \label{eq:s5} &&\\
            &\texttt{sizeOf}(\texttt{process}_{jm}) = \texttt{element}(p_{js^{i}_{j}}, \texttt{workersOf}_{jm}) &&\forall j\in J, i\in \{1, ..., |S_{j}|\} \label{eq:s6} &&\\
            &\text{if }\texttt{workersOf}_{jm} = w \rightarrow \texttt{presenceOf}(\texttt{allocateWorkers}_{jm^{i}_{j}w}) = 1 &&\forall j\in J, i\in \{1, ..., |S_{j}|\}, w\in \{w^{-}_{s^{i}_{j}}, ..., w^{+}_{s^{i}_{j}}\} \label{eq:s7} &&\\
            &\texttt{noOverlap}([\texttt{process}_{jm}|j\in J: m\in j^{*}]) &&\forall m\in M \label{eq:s8} &&\\
            &\sum_{j\in J:m\in j^{*}}\texttt{pulse}(\texttt{waitBefore}_{jm}, 1)\leq R^{-}_{m}&& \forall m\in M \label{eq:s9} &&\\
            &\sum_{j\in J:m\in j^{*}}\texttt{pulse}(\texttt{waitAfter}_{jm}, 1)\leq R^{+}_{m}&& \forall m\in M \label{eq:s10} &&\\
            &\sum_{j\in J}\sum_{i\in \{1, ..., |S_{j}|\}}\sum_{w = w^{-}_{s^{i}_{j}}}^{w^{+}_{s^{i}_{j}}}\texttt{pulse}(\texttt{allocateWorkers}_{jm^{i}_{j}w}, w)\leq W &&\label{eq:s11}&&\\
            & &&\notag &&\\
            &\texttt{waitBefore}_{jm}: \texttt{interval variables} &&\forall j\in J, m\in j^{*} \notag &&\\
            &\texttt{process}_{jm}: \texttt{interval variables} &&\forall j\in J, m\in j^{*} \notag &&\\
            &\texttt{waitAfter}_{jm}: \texttt{interval variables} &&\forall j\in J, m\in j^{*} \notag &&\\
            &\texttt{allocateWorkers}_{jm^{i}_{j}w}: \texttt{interval variables} &&\notag &&\\
                            &\qquad \qquad \texttt{sizeOf}(\texttt{allocateWorkers}_{jm^{i}_{j}w}) = p_{js^{i}_{j}w}&&\notag &&\\
                            &\qquad \qquad \texttt{optional} = \texttt{True} &&\forall j\in J, i\in \{1, ..., |S_{j}|\}, w\in \{w^{-}_{s^{i}_{j}}, ..., w^{+}_{s^{i}_{j}}\} \notag &&\\
            &\texttt{workersOf}_{jm^{i}_{j}}\in \{w^{-}_{s^{i}_{j}}, ..., w^{+}_{s^{i}_{j}}\} &&\forall j\in J, i\in \{1, ..., |S_{j}|\} \notag && 
        \end{flalign}
        \endgroup

        Each operation of a job is connected with three interval variables, indicating the waiting in an entry-point buffer, the processing on the machine and the waiting in an exit-point buffer, namely $\texttt{waitBefore}_{jm}$, $\texttt{process}_{jm}$ and $\texttt{waitAfter}_{jm}$ respectively for each $j$ and $m\in j^{*}$. Constraints (\ref{eq:s2}) and (\ref{eq:s3}) impose the appropriate order of the three variables. Since $\texttt{waitAfter}$ is always the last time interval of an operation, the value of makespan is determined by the maximum end time of these variables (\ref{eq:s1}).

        Constraints (\ref{eq:s4}) ensure that the transition between consecutive machines respects the exact transportation time. Specifically, the predicate $\texttt{startAtEnd}$ enforces that the start time of the second interval variable is equal to the end time of the first, minus a fixed offset corresponding to the transportation time between the two machines. It is essential that the time difference between two consecutive machines equals the transportation time - not merely exceeds it - so that any additional delay is explicitly captured by the variables $\texttt{waitBefore}$ and $\texttt{waitAfter}$, which are subject to resource constraints.

        Optional interval variables $\texttt{allocateWorkers}_{jmw}$ are present if $w$ workers are allocated to the operation $j$ being processed on $m$. The size of variables $\texttt{allocateWorkers}_{jmw}$ is set to the resource-dependent processing time $p_{jsw}$ of the corresponding stage $s$. Constraints (\ref{eq:s5}) ensure that the present interval variable $\texttt{allocateWorkers}_{jmw}$ and $\texttt{process}_{jm}$ are synchronised (i.e., they have the same start and end times). To enhance the performance of the CP model, we also add a set of integer variables $\texttt{workersOf}_{jm}$, indicating the number of assigned workers of $j$ on $m$. Constraints (\ref{eq:s6}) use the predicate $\texttt{element}$ which returns the value of a list, indicating the index of the list by a variable. In this case, (\ref{eq:s6}) define the duration of $\texttt{process}_{jm}$ as the processing time $p_{jsw}$ of the corresponding stage $s$, for which $w = \texttt{workersOf}_{jm}$. By (\ref{eq:s7}), if the value of $\texttt{workersOf}_{jm}$ is set to $w$, the corresponding interval variable $\texttt{allocateWorkers}_{jmw}$ is present.

        Constraints (\ref{eq:s8}) ensure that each machine can process one job at the time. \texttt{Pulse} constraints (\ref{eq:s9}) and (\ref{eq:s10}) satisfy the capacities of buffers on the entry and exit points of each machine respectively. Finally, constraints (\ref{eq:s11}) restrict the number of simultaneously occupied workers to $W$.

    \subsection{Lower bound via a malleable scheduling relaxation} \label{sec:malleable}
Beyond the lower bounds $\text{LB}_{1}$–$\text{LB}_{7}$, functions modelling resource-dependent processing times may allow for a reduction of a relaxed problem to a general case of \emph{malleable scheduling}, under realisting assumptions.

Let $J^{*}$ denote the set of all operations, defined as $J^{*} = \{(j, s^{i}_{j})|j\in J, i\in \{1, ..., |S_{j}|\}$. Also, let $W^{*} = \{1, ..., W\}$ represent the set of workers, modeled as parallel resources. If we disregard the precedence constraints between operations of the same job, treating each operation as an independent task, and temporarily ignore buffer limitations, the relaxed problem  becomes a Parallel Machine Scheduling Problem (PMSP), where $W^{*}$ is the set of parallel machines. Indeed, these assumptions relax the original problem, as some of its constraints are disregarded without adding new ones.

To capture the \emph{parallelisability} of operations across multiple workers, the PMSP reformulation must allow job splitting. However, all partitions of a given operation must be processed simultaneously, sharing identical start and end times across the assigned workers. An example of such parallelisable operations is illustrated in Figure \ref{fig:malleable}.

        \begin{figure}[h]
            \centering
            \includegraphics{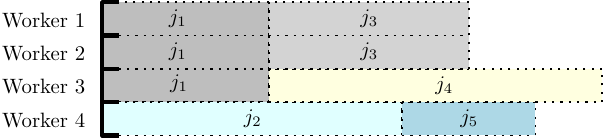}
            \caption{Parallelisability of operations to workers}
            \label{fig:malleable}
        \end{figure}

But then, we notice that the proposed reformulation is equivalent to the \emph{malleable job scheduling} problem \cite{Du89}, in which each job can be assigned to an arbitrary subset of the available machines and is processed non-preemptively and in unison, i.e., with a common start and completion time across them. Most of the literature on malleable job scheduling considers identical machines, where the processing time of a job is a function only of the number of allocated machines. This assumption also holds in our setting:  each operation $j^{\prime}\in J^{*}$ is associated with a processing time function $f_{j^{\prime}}(Q)$, where $Q\subseteq W^{*}$ denotes the subset of workers assigned to that operation. Specifically, $f_{j^{\prime}}(Q) = p_{j^{\prime}|Q|}$ for any operation $j^{\prime}\in J^{*}$ and any subset $Q\in W^{*}$ of size $|Q|$. Under this framework, it is commonly assumed that $f_{j^{\prime}}(Q)$ has the following two properties:

        \noindent\textbf{P.1. Non-increasing processing times.} Processing time $f_{j^{\prime}}(Q)$ is non-increasing in the number of allocated workers $|Q|$.\\
        \textbf{P.2. Monotone work.} The product $|Q|\cdot f_{j^{\prime}}(Q)$ is non-decreasing.

Although the second property looks less straightforward than the first, P.2 captures the fact that the total amount of work required to complete a job does not decrease through parallelisation: to the contrary, parallel execution often entails additional communication and coordination overhead. This observation appears to be strongly applicable and is also verified by the gear-motor plant motivating our study.

Now, inspired by the LP suggested by \cite{fotakis25malleable}, we consider the following relaxation of the malleable job scheduling problem:
        \begin{flalign}
            \text{(LP): }& &&\notag &&\\
            \text{min }&C_{max} &&\notag &&\\
            &\sum_{w\in W^{*}}x_{j^{\prime}w}\geq 1 &&\forall j^{\prime}\in J^{*} \label{eq:lp1}&&\\
            &C_{max}\geq \sum_{j^{\prime}\in J^{*}}p_{j^{\prime}1}\cdot x_{j^{\prime}w} &&\forall w\in W^{*} \label{eq:lp2} &&\\
            &C_{max}\geq p_{js^{i}_{j}w^{+}_{s}} &&\forall j\in J, i\in \{1, ..., |S_{j}|\} \label{eq:lp3} &&\\
            &x_{j^{\prime}w}\geq 0 &&\forall j\in J^{*}, w\in W^{*} \notag &&\\
            &C_{max}\geq 0 &&\notag &&
        \end{flalign}

        \begin{lemma}
            The optimal solution of (LP) provides a lower bound on the makespan of any feasible solution to the malleable job scheduling problem. \label{lem:proof}
        \end{lemma}
        \begin{proof}
            Constraints (\ref{eq:lp1}) ensure that each operation is fully assigned to workers, i.e., the total fraction of an operation assigned across all workers equals one.

            Property P.1 guarantees that Constraint (\ref{eq:lp3}) holds: when an operation $(j, s^{i}_{j})$ is split among multiple workers, its processing time decreases as the number of assigned workers increases, reaching its minimum when the operation is split among the maximum allowed number $w^{+}_{s}$. Hence, the makespan of any feasible schedule must be at least as large as the processing time corresponding to the maximum possible split of any operation.

            By property P.2, the product $|Q| \cdot f_{j'}(Q)$ is non-decreasing with $|Q|$, for any subset of workers $Q$, where $|Q|\geq 1$ denotes the number of workers assigned to operation $j'$. This implies that:
            \begin{equation}
                p_{j'|Q|} \geq \frac{p_{j'1}}{|Q|} \label{eq-work}
            \end{equation}
            
            \medskip
            
            Let us now consider a feasible schedule $\texttt{S}$ of the malleable scheduling problem, where each operation $j^{\prime}$ is assigned to a subset of workers $Q_{j^{\prime}}$ and let $\texttt{S}_w$ be the subset of operations assigned to each worker $w$, and $C^{\texttt{S}}_{max}$ be the makespan of $\texttt{S}$. It must hold that the total load of each worker is bounded above by $C^{\texttt{S}}_{max}$:
             \begin{flalign}
                &C^{\texttt{S}}_{max} \geq \sum_{j' \in \texttt{S}_w}p_{j'|Q_{j^{\prime}}|} && \forall w\in W^*  && && && && &&\label{eq-work1}
            \end{flalign}
            A feasible solution to (LP), denoted by $\bar{\texttt{S}}$, can be defined by setting, for each operation $j^{\prime}\in J^{*}$ and each worker $w\in Q_{j^{\prime}}$ in $\texttt{S}$:
            \[
            \bar{x}_{j'w} =
            \begin{cases}
            \frac{1}{|Q_{j^{\prime}}|}, & \text{if } x_{j^{\prime}w} = 1, \\
            0, & \text{otherwise}.
            \end{cases}
            \]
            
          By substituting solution $\bar{x}_{j'w}$ to Constraint (\ref{eq:lp2}), the total load of each worker $w$ in $\bar{\texttt{S}}$ becomes equal to $\sum_{j' \in \texttt{S}_{w}} \frac{p_{j'1}}{|Q_{j^{\prime}}|}$, and less or equal to the makespan, $C_{max}$, of $\bar{\texttt{S}}$. Note that, for a worker $w^{\prime}$, where the maximum load is achieved, $C_{max} = \sum_{j^{\prime}\in \bar{\texttt{S}}_{w^{\prime}}} \frac{p_{j'1}}{|Q_{j^{\prime}}|}$, where $\bar{\texttt{S}}_{w^\prime}$ is the subset of operations assigned to worker $w^{\prime}$ in $\bar{S}$. Applying (\ref{eq-work}), combined with (\ref{eq-work1}), we yield that: 
           \begin{flalign}
            &C^{\texttt{S}}_{max} \geq \sum_{j'\in \texttt{S}_{w^{\prime}}} p_{j'|Q_{j^{\prime}}|} \geq C_{max}, && \notag && && && && &&
        \end{flalign} 

            

      
       Let $C^*_{max}$ optimal objective value of the (LP). Since $C_{max}$ is the value of a feasible solution to (LP), it should hold that:
          \begin{flalign}
            &C^{\texttt{S}}_{max}\geq C^{*}_{max}, && \notag && (\text{LB}_{8})&& && && &&
        \end{flalign}  
        for any feasible solution $\texttt{S}$ to the malleable job scheduling problem.
        \end{proof}
        Note that $C_{max}^{*}$ of (LP) constitutes the eighth and final lower bound on the makespan. 
    
    \subsection{CP-oriented LBBD}
        Employing a CP-formulated subproblem within a LBBD algorithm is a common approach in exact decomposition methods. This design is particularly well-suited for scheduling problems with renewable resources, where CP can efficiently exploit the \texttt{Cumulative} predicate to handle resource usage over time. In contrast, using a CP formulation as the master problem is typically avoided, since CP solvers are generally less effective at producing tight lower bounds, which is a critical requirement for ensuring rapid convergence and small optimality gaps in LBBD. However, the relaxation defined by $\mathcal{M}$, strengthened by lower bounds $(\text{LB}_{1})-(\text{LB}_{7})$ and the (LP) relaxation of the reduced malleable scheduling problem, results in a strong and efficient master problem formulation using this time only CP.

As observed, the subproblem receives as input a fixed assignment of jobs to machines, but without specifying their processing order. In other words, the hybrid flexibility of the original problem, where a job can be assigned to any eligible machine at each stage, is restricted in the subproblem: each stage of a job has a dedicated machine by the master problem, while the exact scheduling is determined from scratch in the subproblem. To ensure progress in the LBBD algorithm, a valid Benders cut must prevent the master problem from reassigning all jobs to the same machines as in a previously explored solution, unless no improved solution (i.e., one with a better lower bound than the current best objective value) can be found. This guarantees that the master explores new assignments unless repeating one is provably necessary.

        For a set of job-machine sequences $J^*$, in which each $j^{*} = (m^{1}_{j}, m^{2}_{j}, ..., m^{|S_{j}|}_{j})$ is the order of machines which process job $j$, and $\zeta$ is the corresponding minimum value of makespan, as computed by the subproblem, the following straightforward optimality cut is constructed:

        \begin{flalign}
            &\text{if }\sum_{j^{*}\in J^*}\sum_{i\in \{1, ..., |S_{j}|\}}(\texttt{machineOf}_{js^{i}_{j}} = m^{i}_{j}) = \sum_{j\in J}|S_{j}|\rightarrow C_{max}\geq \zeta && \label{eq:benders-cut}   
        \end{flalign}
or in other words, if all jobs are processed at the same machines, the value of makespan is set to the objective value of the subsequent subproblem. Algorithm \ref{alg:lbbd} presents this iterative procedure.

        \begin{algorithm}[H]
            \caption{CP-oriented Logic-Based Benders Decomposition}
            \label{alg:lbbd}
            \begin{algorithmic}[1]
                \State Set $\text{LB} = 0, \text{UB} = \infty, k = 0$ as the lower bound, the upper bound and the number of iteration;
                \State Solve $\text{(LP)}\rightarrow \text{LB}$
                \While{$\text{LB} < \text{UB}$}
                    \State Solve $\mathcal{M}\rightarrow \text{LB}, J^*$, $J^*$ being the sequences of machines per job;
                    \State Solve $\mathcal{S}$ for $J^*\rightarrow \zeta^{k}$ the objective value at iteration $k$;
                    \If{$\zeta^{k} < \text{UB}$}
                        \State $\zeta^{k}\rightarrow \text{UB}$;
                    \EndIf
                    \State Add Benders cut (\ref{eq:benders-cut}) to $\mathcal{M}$;
                    \State $k\rightarrow k + 1$;
                \EndWhile
                \State Terminate;
            \end{algorithmic}
        \end{algorithm}

\section{Computational experience} \label{sec:experiments}
    All experiments run on a server equipped with 32 AMD Ryzen Threadripper PRO 5955WX 16-core processors and 32 GB of RAM, running Ubuntu 22.04.5 LTS. All mathematical models (LPs for the lower bounds and CP models) are solved by employing the \emph{CPLEX 22.1.1} optimiser. The LPs have been formulated using \emph{Pyomo 6.7.0}. For the complete $\mathcal{CP}$ model, the master problem and the subproblem, we have also used the 
    DOCplex module in \emph{Python 3.10.12}.

    For the $\mathcal{CP}$ model, a time limit of 30 minutes is imposed. In the case of the LBBD algorithm, a 5-minute time limit is set for the master problem, given that larger datasets may not converge to optimality quickly. To ensure a fair comparison between the two methods, the overall execution time for the LBBD algorithm is also capped at 30 minutes. While this time limit may appear short, it is sufficient to highlight the comparative performance of the methods. Empirical observations indicate that the convergence is slower for larger time limits, from which no additional meaningful insights can be drawn.  
    
    \subsection{Applied testbeds}
        We perform two distinct experimental setups to assess the performance of the methods under different testbeds. For the first group of experiments, we use the datasets provided by \cite{armstrong2021hybrid}, which include instances with 20, 50, 100, 200, and 400 jobs. These testbeds are based on a problem variant closely related to ours, as they incorporate hybrid job-to-machine flexibility across stages and account for transportation times between machines of consecutive stages. Moreover, the high degree of machine parallelism (10 machines per stage) poses additional challenges for the LBBD scheme, since the Benders cuts defined in (\ref{eq:benders-cut}) can be easily overcome by assigning jobs to alternative machines.
        
        Each one of the eight stages has 10 parallel machines. Stages 4 and 8 are skipped by some jobs. Since the original problem in \cite{armstrong2021hybrid} does not account for restricted buffers, we introduced additional parameters to adapt the testbeds to our problem requirements. For all machines, the number of buffers at the entry and exit points is drawn from a discrete uniform distribution between 1 and 5. This rule excludes machines that belong to the initial or final stages of any job's processing; in those cases, the number of entry or exit buffers is set to $|J|$ - effectively infinite in practice. In all datasets, we consider a total of 20 workers. Each job can be processed by 1, 2, or 3 workers.

        For the second group of experiments, we generate 30 datasets based on the random distributions proposed by \cite{naderi2014hybrid}. Compared to the first group, these datasets have significantly fewer stages and machines per stage. Consequently, the number of buffers and available workers is also reduced, while the range of processing times remains similar for both groups.

        Each generated instance includes 20, 50, 100, 200, or 400 jobs, with the number of stages varying between 2 and 4, as suggested by \cite{naderi2014hybrid}. The nominal processing times (i.e., for one worker) follow an integer distribution between 2 and 15 time units. The total number of available workers is set to 8 for all datasets, regardless of the number of jobs. Each job can be processed by 1, 2, or 3 workers.

        For each combination of job count (20-400) and stage count (2-4), we create two variants: two machines at each stage, and a random number of machines per stage, between 1 and 3. Each job has a 20\% probability of skipping a stage, enforcing processing for at least one stage. For each machine, the number of entry and exit buffers is randomly chosen between 1 and 3. Additionally, for each machine pair $m, n$ where $n$ follows $m$, transportation time is drawn uniformly from 1 to 9 units, as in \cite{armstrong2021hybrid}, and multiplied by the stage difference (e.g., if $m$ is in stage 1 and $n$ in stage 3, the random value is multiplied by 2). In total, 30 datasets are generated, corresponding to the combinations of five job counts, three stage counts, and two machine allocation variants (fixed at two machines per stage or random between one and three).
        
         Since none of the aforementioned generators consider processing times which are dependent on the number of allocated workers, we assume that processing times are obtained by a function $f_{j'}(Q) = \frac{p_{j'1}}{|Q|}$, i.e., the division of the nominal processing time of operation $j'$ by the number of allocated workers $|Q|$. Function $f_{j'}(\cdot)$ has been selected as it satisfies the two properties P.1 and P.2 for which the relaxation of the malleable job scheduling problem can be directly applied. More precisely:

        \begin{lemma}
            Function $f_{j^\prime}(Q) = \frac{p_{j^{\prime}1}}{|Q|}$ is non-increasing on the number of allocated workers $|Q|$. \label{lem:1}
        \end{lemma}
        \begin{proof}
            Let $W^{*}$ be the set of workers and $Q_1, Q_2$ two non-empty subsets $Q_1, Q_2\subseteq W^*$, where $|Q_1|\leq |Q_2|$.  Then $f_{j'}(Q_1) = \frac{p_{j'1}}{|Q_1|}\geq \frac{p_{j'1}}{|Q_2|} = f_{j'}(Q_2)$, and the lemma holds.
        \end{proof}

        \begin{lemma}
            The product $|Q|\cdot f_{j'}(Q)$ is non-decreasing on the number of allocated workers $|Q|$. \label{lem:2}
        \end{lemma}
        \begin{proof}
            By construction, the product $|Q|\cdot f_{j'}(Q)$ is equal to $p_{j'1}$, thus constant and trivially non-decreasing.
        \end{proof}
        \noindent By Lemma \ref{lem:1} and \ref{lem:2}, $\text{LB}_{8}$ is valid for the generated datasets.

    \subsection{Comparison between the CP and the LBBD}
        All 155 datasets (125 for the first group and 30 for the second group) are solved using both the $\mathcal{CP}$ model and the CP-oriented LBBD approach, with a time limit of 1,800 seconds. The performance of the two methods is evaluated based on the number of instances for which a feasible solution is found and the optimality gap at the time limit, primarily determined by the best upper bound, as CP is known to be inefficient at producing strong lower bounds.
        
        This inefficiency is also evident in the experimental results. We report the best lower bounds provided by each method individually, as well as the strongest lower bound among $\text{LB}_{1}$ to $\text{LB}_{8}$, which is computed once at the beginning of the solution process.
        
        \textbf{First group of instances. }Table \ref{tab:results} of the Appendix reports the detailed results, and table \ref{tab:average} summarises the average values per number of jobs.

        \begin{table}[H]
            \centering
            \caption{Average results per number of jobs}
            \label{tab:average}
            \resizebox{\textwidth}{!}{
            \begin{tabular}{|cc|ccccc|ccccc|}
            \hline
            \multirow{2}{*}{$|J|$} & \multirow{2}{*}{Best LB} & \multicolumn{5}{|c|}{$\mathcal{CP}$} & \multicolumn{5}{c|}{LBBD} \\ \cline{3-12} 
                                 &                          & LB    & UB     & Original Gap (\%) & Real Gap (\%) & Solved & LB    & UB     & Original Gap (\%) & Real Gap (\%) & Solved \\ \hline
            20                   & 44.00                       & 21.64 & 48.52  & 55.30        & 9.30     & 25     & 26.72 & 45.76  & 41.31        & 3.85     & 25     \\
            50                   & 43.56                    & 22.52 & 52.56  & 57.10        & 17.09    & 25     & 24.48 & 50.2   & 51.20        & 13.23    & 25     \\
            100                  & 52.12                    & 23.32 & 84.48  & 72.01        & 37.52    & 25     & 35.68 & 60.2   & 40.68        & 13.36    & 25     \\
            200                  & 75.92                    & 22.76 & 141.94 & 82.26        & 44.01    & 16     & 58.52 & 92.32  & 36.59        & 17.72    & 25     \\
            400                  & 123.28                   & -     & -      & -            & -        & 0      & 96.52 & 156.80 & 38.42        & 21.35    & 25 \\ \hline
            \end{tabular}
            }
        \end{table}

        Column `Best LB' reports the average of the best value among $\text{LB}_{1} - \text{LB}_{8}$. Columns `LB' indicate the original lower bounds of the methods, and `UB' report the best found upper bounds. `Original Gap' corresponds to the optimality gap which is provided directly by each method: $\text{Original Gap} = \frac{\text{UB} - \text{LB}}{\text{UB}}\%$. `Real Gap' replaces the original lower bound with the best bound among columns `Best LB' and `LB': $\text{Real Gap} = \frac{\text{UB} - max\{\text{Best LB}, \text{LB}\}}{\text{UB}}\%$.

        As observed, the $\mathcal{CP}$ model yields feasible solutions for all instances with 20 to 100 jobs and for most instances with 200 jobs (16 out of 25). However, it fails to provide any feasible solutions for the largest testbeds (400 jobs). The scalability challenge is also related to the lower bound: the average values of the initial lower bound remain similar across all instance sizes, despite the increasing number of jobs. Consequently, the initial optimality gap increases significantly - from 55.30\% for 20-job instances to 82.26\% for the largest solvable instances (200 jobs).

        In contrast, the LBBD approach produces feasible solutions for all 125 instances. Although the initial lower bounds remain weak, resulting in large gaps from optimality (ranging from 36.59\% to 51.20\%), the method improves these bounds as the instance size increases, resulting in more consistent performance across all scales. In terms of upper bounds, LBBD consistently outperforms the complete model. Notably, after introducing the tighter lower bounds $\text{LB}_{1}$ to $\text{LB}_{8}$, the resulting upper bounds for 20-job instances are near-optimal, with an average gap of just 3.85\%. For larger instances, the gap increases with problem size but remains relatively stable, ranging from 13.23\% to 21.35\%. While these values are not strictly near-optimal, they demonstrate the strong efficiency of the method given the complexity and scale of the problem.

        These observations are further illustrated by the values of the Real Gap for both methods across all 125 instances, as shown in Figure \ref{fig:plotLB}:

        \begin{figure}[H]
        \centering
        \caption{Gap (\%) values of all 125 instances}
        \begin{tikzpicture}
            \begin{axis}[
                height = 9cm,
                width = 15cm,
                ylabel={Gap (\%)},
                legend entries={$\mathcal{CP}$, LBBD},
                legend pos=north west,
                xtick=\empty,
                xmin = 0,
                xmax = 126,
                ymin = 0,
                ymax = 75,
                grid=both
            ]
            \addplot[color=blue, mark=*, only marks] coordinates {
                (1, 8.33) 	(2, 11.76) 	(3, 9.09) 	(4, 10.64) 	(5, 8.33) 	(6, 9.26) 	(7, 8.33) 	(8, 8.33) 	(9, 7.41) 	(10, 8.33) 	(11, 10.2) 	(12, 10.64) 	(13, 10) 	(14, 7.84) 	(15, 9.09) 	(16, 8.7) 	(17, 11.63) 	(18, 8.33) 	(19, 10.42) 	(20, 6.25) 	(21, 12) 	(22, 7.84) 	(23, 7.32) 	(24, 12.96) 	(25, 9.43) 	(26, 15.69) 	(27, 18.18) 	(28, 16.33) 	(29, 15.69) 	(30, 16.67) 	(31, 19.3) 	(32, 17.31) 	(33, 17.31) 	(34, 17.54) 	(35, 17.31) 	(36, 17.31) 	(37, 17.65) 	(38, 18.18) 	(39, 16.67) 	(40, 19.23) 	(41, 16) 	(42, 18.75) 	(43, 13.46) 	(44, 16.33) 	(45, 13.46) 	(46, 15.38) 	(47, 15.69) 	(48, 20) 	(49, 18.97) 	(50, 18.97) 	(51, 41.76) 	(52, 44.44) 	(53, 40.74) 	(54, 43.02) 	(55, 37.65) 	(56, 37.36) 	(57, 42.05) 	(58, 41.38) 	(59, 39.08) 	(60, 40.23) 	(61, 39.33) 	(62, 39.77) 	(63, 41.76) 	(64, 40) 	(65, 33.33) 	(66, 41.67) 	(67, 36.47) 	(68, 39.08) 	(69, 15.25) 	(70, 44.44) 	(71, 14.29) 	(72, 42.53) 	(73, 39.53) 	(74, 43.16) 	(75, 19.7) 	(76, 38.89) 	(77, 100) 	(78, 39.02) 	(79, 38.89) 	(80, 100) 	(81, 62.96) 	(82, 100) 	(83, 100) 	(84, 100) 	(85, 41.22) 	(86, 100) 	(87, 36) 	(88, 42.75) 	(89, 41.91) 	(90, 100) 	(91, 40.8) 	(92, 41.27) 	(93, 38.17) 	(94, 42.31) 	(95, 100) 	(96, 100) 	(97, 69.67) 	(98, 43.55) 	(99, 41.43) 	(100, 45.26)
            };
            \addplot[color=red, mark=*, only marks] coordinates {
                (1, 2.22) 	(2, 4.26) 	(3, 2.44) 	(4, 4.55) 	(5, 4.35) 	(6, 3.92) 	(7, 4.35) 	(8, 4.35) 	(9, 1.96) 	(10, 4.35) 	(11, 4.35) 	(12, 4.55) 	(13, 4.26) 	(14, 2.08) 	(15, 4.76) 	(16, 4.55) 	(17, 5) 	(18, 2.22) 	(19, 4.44) 	(20, 2.17) 	(21, 6.38) 	(22, 2.08) 	(23, 2.56) 	(24, 6) 	(25, 4) 	(26, 12.24) 	(27, 13.46) 	(28, 12.77) 	(29, 12.24) 	(30, 11.76) 	(31, 13.21) 	(32, 14) 	(33, 14) 	(34, 12.96) 	(35, 14) 	(36, 12.24) 	(37, 14.29) 	(38, 11.76) 	(39, 11.76) 	(40, 16) 	(41, 12.5) 	(42, 15.22) 	(43, 11.76) 	(44, 12.77) 	(45, 11.76) 	(46, 13.73) 	(47, 12.24) 	(48, 14.89) 	(49, 14.55) 	(50, 14.55) 	(51, 10.17) 	(52, 16.67) 	(53, 11.11) 	(54, 15.52) 	(55, 13.11) 	(56, 12.31) 	(57, 13.56) 	(58, 15) 	(59, 15.87) 	(60, 16.13) 	(61, 10) 	(62, 13.11) 	(63, 14.52) 	(64, 16.39) 	(65, 6.67) 	(66, 14.04) 	(67, 5.26) 	(68, 10.17) 	(69, 15.25) 	(70, 16.67) 	(71, 14.29) 	(72, 16.67) 	(73, 8.77) 	(74, 15.63) 	(75, 17.19) 	(76, 18.09) 	(77, 18.68) 	(78, 13.79) 	(79, 15.38) 	(80, 20.43) 	(81, 16.67) 	(82, 21.98) 	(83, 19.15) 	(84, 19.35) 	(85, 18.09) 	(86, 20.43) 	(87, 14.89) 	(88, 21.05) 	(89, 18.56) 	(90, 14.44) 	(91, 16.85) 	(92, 14.94) 	(93, 10) 	(94, 14.77) 	(95, 19.35) 	(96, 19.79) 	(97, 16.85) 	(98, 20.45) 	(99, 18) 	(100, 21.05) 	(101, 19.23) 	(102, 21.52) 	(103, 18.3) 	(104, 19.11) 	(105, 22.44) 	(106, 22.75) 	(107, 24.2) 	(108, 22.09) 	(109, 28.4) 	(110, 21.88) 	(111, 23.38) 	(112, 21.57) 	(113, 24.69) 	(114, 18.01) 	(115, 19.08) 	(116, 20.67) 	(117, 18.24) 	(118, 16.98) 	(119, 16.23) 	(120, 24.36) 	(121, 22.01) 	(122, 20.92) 	(123, 22.22) 	(124, 21.69) 	(125, 23.75) 

            };
            \end{axis}
        \end{tikzpicture}
        \label{fig:plotLB}
        \end{figure}
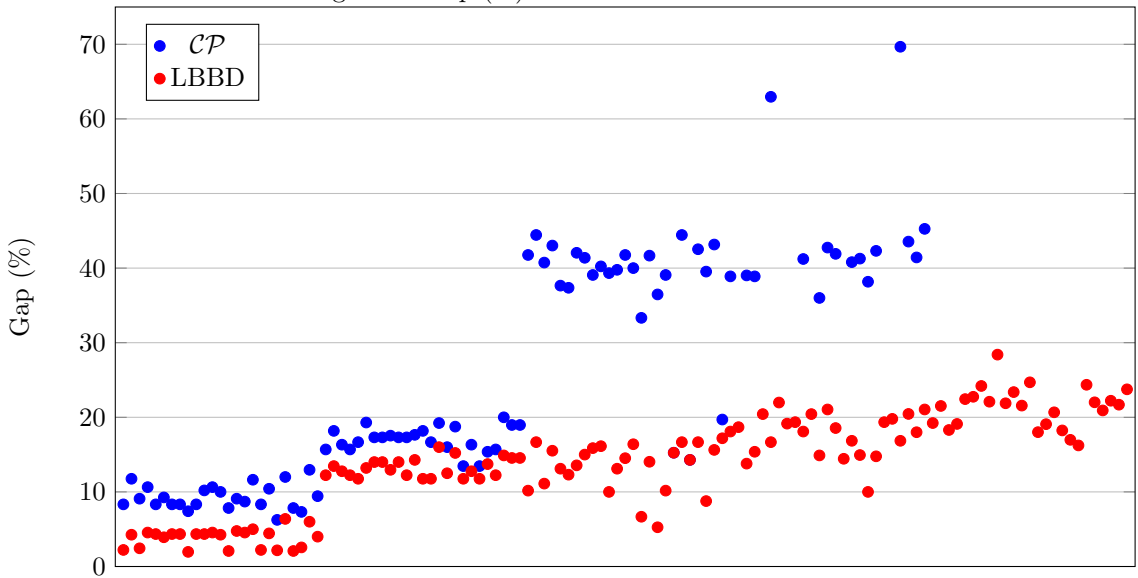

    Table \ref{tab:lb} highlights the impact of the introduced lower bounds ($\text{LB}_{1} - \text{LB}_{8}$) on the performance of both methods. The column labeled `Diff' reports the average relative improvement, calculated as $\frac{\text{Best LB} - \text{LB}}{\text{Best LB}}\%$.

    For the $\mathcal{CP}$ model, the impact of the introduced bounds becomes more pronounced as instance size increases. This is primarily due to the model's growing inefficiency in strengthening dual bounds for larger problems. In contrast, for the LBBD approach, the relative impact of the introduced bounds decreases with instance size. However, the improvement remains substantial across all testbed groups, with the smallest observed improvement still reaching 21.71\%.

    \begin{table}[H]
        \centering
        \caption{Impact of the introduced lower bounds}
        \label{tab:lb}
        \begin{tabular}{cccccc}
        \hline
        \multirow{2}{*}{$|J|$} & \multirow{2}{*}{Best LB} & \multicolumn{2}{c}{$\mathcal{CP}$} & \multicolumn{2}{c}{LBBD} \\ \cline{3-6}
            &        & LB       & Diff  & LB    & Diff  \\ \hline
        20  & 44.00  & 21.64    & 50.09 & 26.72 & 39.27 \\
        50  & 43.56  & 22.52     & 48.12 & 24.48 & 43.80 \\
        100 & 52.12  & 23.32 & 56.31 & 35.68 & 31.54 \\
        200 & 75.92  & 22.76    & 68.39 & 58.52 & 22.92 \\
        400 & 123.28 & -        & -     & 96.52 & 21.71 \\ \hline
        \end{tabular}
    \end{table}

    \textbf{Second group of experiments.}
        Table \ref{tab:second} summarises the results of the second group of experiments. Each instance is named using the format $|J|\_|S|\_\#$, where $|J|$ denotes the number of jobs, $|S|$ the number of stages, and $\#$ the variant (1 for two machines per stage, and 2 for one to three machines per stage). The $\mathcal{CP}$ model and LBBD produced identical upper bounds for 12 out of the 30 datasets, all corresponding to instances with 20 or 50 jobs. For the remaining datasets, $\mathcal{CP}$ failed to find any feasible solution for seven instances with 400 jobs but provided slightly better upper bounds than LBBD for three datasets. In terms of lower bounds, LBBD outperformed $\mathcal{CP}$ in 27 datasets; for the remaining three, both methods achieved equal values. The additional lower bounds incorporated in LBBD are critical, as reflected in the low values of the Real Gap (below 10\% for almost all datasets, including those with 400 jobs). This strongly suggests that the solutions provided by both methods are near-optimal.
        
        \begin{table}[H]
            \centering
            \caption{Results of the experiments of the second group}
            \label{tab:second}
            \resizebox{\textwidth}{!}{
            \begin{tabular}{|cc|cccc|cccc|}
            \hline
            \multirow{2}{*}{Instance} & \multirow{2}{*}{Best LB} & \multicolumn{4}{|c|}{$\mathcal{CP}$} & \multicolumn{4}{|c|}{LBBD}           \\ \cline{3-10}
                              &                          & LB & UB  & Original Gap (\%) & Real Gap (\%) & LB & UB  & Original Gap (\%) & Real Gap (\%) \\ \hline
            20\_2\_1  & 41   & 10  & 42  & 76.19 & 2.38  & 36   & 42   & 14.29 & 2.38  \\
            20\_2\_2  & 32   & 10  & 36  & 72.22 & 11.11 & 32   & 36   & 11.11 & 11.11 \\
            20\_3\_1  & 49   & 13  & 50  & 74.00 & 2.00  & 32   & 50   & 36.00 & 2.00  \\
            20\_3\_2  & 46   & 14  & 47  & 70.21 & 2.13  & 30   & 47   & 36.17 & 2.13  \\
            20\_4\_1  & 76   & 16  & 78  & 79.49 & 2.56  & 45   & 78   & 42.31 & 2.56  \\
            20\_4\_2  & 65   & 52  & 68  & 23.53 & 4.41  & 54   & 68   & 20.59 & 4.41  \\
            50\_2\_1  & 99   & 10  & 102 & 90.20 & 2.94  & 74   & 102  & 27.45 & 2.94  \\
            50\_2\_2  & 110  & 114 & 114 & 0.00  & 0.00  & 114  & 114  & 0.00  & 0.00  \\
            50\_3\_1  & 127  & 15  & 130 & 88.46 & 2.31  & 74   & 131  & 43.51 & 3.05  \\
            50\_3\_2  & 140  & 15  & 142 & 89.44 & 1.41  & 70   & 142  & 50.70 & 1.41  \\
            50\_4\_1  & 172  & 19  & 176 & 89.20 & 2.27  & 73   & 174  & 58.05 & 1.15  \\
            50\_4\_2  & 156  & 16  & 158 & 89.87 & 1.27  & 67   & 158  & 57.59 & 1.27  \\
            100\_2\_1 & 169  & 10  & 179 & 94.41 & 5.59  & 126  & 177  & 28.81 & 4.52  \\
            100\_2\_2 & 250  & 250 & 250 & 0.00  & 0.00  & 250  & 250  & 0.00  & 0.00  \\
            100\_3\_1 & 260  & 15  & 270 & 94.44 & 3.70  & 134  & 271  & 50.55 & 4.06  \\
            100\_3\_2 & 255  & 259 & 273 & 5.13  & 6.59  & 259  & 273  & 5.13  & 6.59  \\
            100\_4\_1 & 349  & 73  & -   & -     & -     & 141  & 363  & 61.16 & 3.86  \\
            100\_4\_2 & 343  & 267 & 368 & 27.45 & 6.79  & 268  & 368  & 27.17 & 6.79  \\
            200\_2\_1 & 349  & 10  & 376 & 97.34 & 7.18  & 264  & 375  & 29.60 & 6.93  \\
            200\_2\_2 & 349  & 10  & 372 & 97.31 & 6.18  & 175  & 370  & 52.70 & 5.68  \\
            200\_3\_1 & 499  & 15  & 532 & 97.18 & 6.20  & 260  & 542  & 52.03 & 7.93  \\
            200\_3\_2 & 479  & 186 & -   & -     & -     & 456  & 518  & 11.97 & 7.53  \\
            200\_4\_1 & 676  & 248 & -   & -     & -     & 268  & 714  & 62.46 & 5.32  \\
            200\_4\_2 & 683  & 192 & -   & -     & -     & 272  & 744  & 63.44 & 8.20  \\
            400\_2\_1 & 700  & 10  & 768 & 98.70 & 8.85  & 519  & 767  & 32.33 & 8.74  \\
            400\_2\_2 & 705  & 10  & 773 & 98.71 & 8.80  & 356  & 762  & 53.28 & 7.48  \\
            400\_3\_1 & 1007 & 12  & -   & -     & -     & 505  & 1119 & 54.87 & 10.01 \\
            400\_3\_2 & 1021 & 306 & -   & -     & -     & 1003 & 1092 & 8.15  & 6.50  \\
            400\_4\_1 & 1294 & 321 & -   & -     & -     & 501  & 1457 & 65.61 & 11.19 \\
            400\_4\_2 & 1381 & 341 & -   & -     & -     & 1044 & 1591 & 34.38 & 13.20 \\ \hline
            \end{tabular}
            }
        \end{table}

        The second group of experiments allows us to examine the impact of the number of stages on method performance. To this end, we compare the average `Original Gap' (i.e., without considering the additional bounds $\text{LB}_{1} - \text{LB}_{8}$) values of the two methods for each stage count. Table \ref{tab:stages} reports these averages, grouped by the number of stages $|S|$, along with the column `Solved', which indicates the number of datasets, from a total of 10 per number of stages, for which a feasible solution was obtained.

        For $\mathcal{CP}$, the average gaps for instances with three and four stages exclude unsolved datasets. In contrast, LBBD provided solutions for all datasets; therefore, two averages are reported: `Average Gap 1' presents the average gaps for only over datasets where both methods found feasible solutions, ensuring a fair comparison, and `Average Gap 2' presents the average gaps over all datasets. The observed differences in gaps are primarily due to the weaker lower bounds provided by $\mathcal{CP}$, whereas LBBD achieved more reliable bounds.

        \begin{table}[H]
            \centering
            \caption{Average original gaps per number of stages}
            \label{tab:stages}
            \begin{tabular}{|c|cc|ccc|}
            \hline
            \multirow{2}{*}{$|S|$} & \multicolumn{2}{|c|}{$\mathcal{CP}$} & \multicolumn{3}{|c|}{LBBD} \\ \cline{2-6}
            & Average Gap (\%) & Solved & Average Gap 1 (\%) & Average Gap 2 (\%) & Solved \\ \hline
            2                    & 72.51                 & 10         & 24.96              & 24.96              & 10     \\
            3                    & 74.12                 & 7          & 39.16              & 34.91              & 10     \\
            4                    & 61.91                 & 5          & 41.14              & 49.28              & 10 \\ \hline   
            \end{tabular}
        \end{table}

        We compare the Original Gaps of the two methods across the two machine allocation variants: Variant 1, with a fixed number of two machines per stage, and Variant 2, with a random number of machines between one and three. Table \ref{tab:machines} reports the average gaps (calculated only over datasets with feasible solutions) and the number of solved instances per variant (from a total of 15 per variant), similar to Table \ref{tab:stages}. The results reveal a substantial difference between the two variants: instances with one to three machines per stage (Variant 2) exhibit significantly smaller gaps compared to those with two machines per stage (Variant 1).

        \begin{table}[H]
            \centering
            \caption{Average gap per variant}
            \label{tab:machines}
            \begin{tabular}{|c|cc|ccc|}
            \hline
            \multirow{2}{*}{Variant} & \multicolumn{2}{|c|}{$\mathcal{CP}$} & \multicolumn{3}{|c|}{LBBD}                         \\ \cline{2-6}
                                     & Average Gap (\%)      & Solved     & Average Gap 1 (\%) & Average Gap 2 (\%) & Solved \\ \hline
            1                        & 89.06                 & 11         & 37.72              & 43.94              & 15     \\
            2                        & 52.17                 & 11         & 28.59              & 28.83              & 15 \\ \hline
            \end{tabular}
        \end{table}

        Tables \ref{tab:stages} and \ref{tab:machines} highlight how dataset attributes influence method performance. For $\mathcal{CP}$, increasing the number of stages reduces the likelihood of obtaining a solution within the time limit: the transition from two to three stages, and from three to four, leads to fewer feasible solutions (Table \ref{tab:stages}). In contrast, LBBD consistently solves all datasets. Conversely, the number of machines per stage does not affect feasibility, as both variants yield the same number of solved instances (Table \ref{tab:machines}).

        Datasets in Variant 2 include stages that create significant bottlenecks, as any stage with only one machine can substantially tighten the lower bound. An examination of Table \ref{tab:second} indicates that the observed differences in gaps (Table \ref{tab:machines}) are primarily due to higher lower bounds in Variant 2. This effect applies to both methods; however, LBBD demonstrates substantially superior performance compared to the $\mathcal{CP}$ model.
\section{Further work} \label{sec:conclusions}
This work examines flowshop scheduling by considering flexibility across stages and machines, as well as incorporating buffers, modeled as renewable resources. In addition to constraints already explored in the literature, the key challenge addressed here is the dependency of processing times on the number of workers allocated per job and machine. While a CP model can solve instances of substantial size, a CP-oriented LBBD algorithm demonstrates superior performance, yielding low and stable optimality gaps even for instances with hundreds of jobs. A critical factor in this performance is the generation of strong lower bounds through extensions of state-of-the-art dual bounds and problem reduction to a malleable job scheduling relaxation.

 The proposed formulations are adaptable to a range of industrial constraints, including machine or stage-specific limits on worker allocation, machine eligibility restrictions, and the inclusion of non-renewable resources. Moreover, the reduction to malleable scheduling applies under a broad class of processing time functions, thus enabling the transferability of our approach to additional industrial settings. In addition, our method could be used as a simulation tool to examine alternative line configurations and thus offer meaningful managerial insights.

From a methodological perspective, metaheuristics offer a promising prospect for handling increased problem scale and complexity, although our method already scales to daily instances of an actual plant. In particular, hybrid approaches that partition the problem, such as replacing the master problem with a fast metaheuristic to assign jobs to machines, could significantly enhance computational efficiency. All these alternative solution methods could require a detailed computational investigation, which in turn would benefit from the availability of additional benchmark instances.
    



\setstretch{1.1}
\printbibliography
\pagebreak
\section*{Appendix: Results}
    \begin{longtable}[c]{cccccccccc}
\caption{Detailed results of the experiments of the first group}
\label{tab:results} \\
\multirow{2}{*}{Instance} & \multirow{2}{*}{Best LB} & \multicolumn{4}{c}{$\mathcal{CP}$} & \multicolumn{4}{c}{LBBD}           \\
                          &                          & LB & UB  & Original Gap & Real Gap & LB & UB  & Original Gap & Real Gap \\
\endfirsthead
\endhead
20\_1   & 44  & 22 & 48  & 54.17 & 8.33  & 29  & 45  & 35.56 & 2.22  \\
20\_2   & 45  & 20 & 51  & 60.78 & 11.76 & 22  & 47  & 53.19 & 4.26  \\
20\_3   & 40  & 21 & 44  & 52.27 & 9.09  & 28  & 41  & 31.71 & 2.44  \\
20\_4   & 42  & 20 & 47  & 57.45 & 10.64 & 27  & 44  & 38.64 & 4.55  \\
20\_5   & 44  & 22 & 48  & 54.17 & 8.33  & 29  & 46  & 36.96 & 4.35  \\
20\_6   & 49  & 22 & 54  & 59.26 & 9.26  & 22  & 51  & 56.86 & 3.92  \\
20\_7   & 44  & 23 & 48  & 52.08 & 8.33  & 30  & 46  & 34.78 & 4.35  \\
20\_8   & 44  & 21 & 48  & 56.25 & 8.33  & 23  & 46  & 50.00 & 4.35  \\
20\_9   & 50  & 24 & 54  & 55.56 & 7.41  & 31  & 51  & 39.22 & 1.96  \\
20\_10  & 44  & 20 & 48  & 58.33 & 8.33  & 21  & 46  & 54.35 & 4.35  \\
20\_11  & 44  & 25 & 49  & 48.98 & 10.20 & 32  & 46  & 30.43 & 4.35  \\
20\_12  & 42  & 21 & 47  & 55.32 & 10.64 & 28  & 44  & 36.36 & 4.55  \\
20\_13  & 45  & 23 & 50  & 54.00 & 10.00 & 30  & 47  & 36.17 & 4.26  \\
20\_14  & 47  & 21 & 51  & 58.82 & 7.84  & 21  & 48  & 56.25 & 2.08  \\
20\_15  & 40  & 18 & 44  & 59.09 & 9.09  & 19  & 42  & 54.76 & 4.76  \\
20\_16  & 42  & 20 & 46  & 56.52 & 8.70  & 27  & 44  & 38.64 & 4.55  \\
20\_17  & 38  & 21 & 43  & 51.16 & 11.63 & 28  & 40  & 30.00 & 5.00  \\
20\_18  & 44  & 22 & 48  & 54.17 & 8.33  & 29  & 45  & 35.56 & 2.22  \\
20\_19  & 43  & 23 & 48  & 52.08 & 10.42 & 30  & 45  & 33.33 & 4.44  \\
20\_20  & 45  & 22 & 48  & 54.17 & 6.25  & 28  & 46  & 39.13 & 2.17  \\
20\_21  & 44  & 24 & 50  & 52.00 & 12.00 & 31  & 47  & 34.04 & 6.38  \\
20\_22  & 47  & 21 & 51  & 58.82 & 7.84  & 23  & 48  & 52.08 & 2.08  \\
20\_23  & 38  & 20 & 41  & 51.22 & 7.32  & 27  & 39  & 30.77 & 2.56  \\
20\_24  & 47  & 22 & 54  & 59.26 & 12.96 & 25  & 50  & 50.00 & 6.00  \\
20\_25  & 48  & 23 & 53  & 56.60 & 9.43  & 28  & 50  & 44.00 & 4.00  \\
50\_1   & 43  & 21 & 51  & 58.82 & 15.69 & 24  & 49  & 51.02 & 12.24 \\
50\_2   & 45  & 23 & 55  & 58.18 & 18.18 & 25  & 52  & 51.92 & 13.46 \\
50\_3   & 41  & 22 & 49  & 55.10 & 16.33 & 24  & 47  & 48.94 & 12.77 \\
50\_4   & 43  & 21 & 51  & 58.82 & 15.69 & 24  & 49  & 51.02 & 12.24 \\
50\_5   & 45  & 23 & 54  & 57.41 & 16.67 & 24  & 51  & 52.94 & 11.76 \\
50\_6   & 46  & 24 & 57  & 57.89 & 19.30 & 26  & 53  & 50.94 & 13.21 \\
50\_7   & 43  & 24 & 52  & 53.85 & 17.31 & 24  & 50  & 52.00 & 14.00 \\
50\_8   & 43  & 21 & 52  & 59.62 & 17.31 & 24  & 50  & 52.00 & 14.00 \\
50\_9   & 47  & 24 & 57  & 57.89 & 17.54 & 26  & 54  & 51.85 & 12.96 \\
50\_10  & 43  & 20 & 52  & 61.54 & 17.31 & 24  & 50  & 52.00 & 14.00 \\
50\_11  & 43  & 25 & 52  & 51.92 & 17.31 & 25  & 49  & 48.98 & 12.24 \\
50\_12  & 42  & 23 & 51  & 54.90 & 17.65 & 24  & 49  & 51.02 & 14.29 \\
50\_13  & 45  & 23 & 55  & 58.18 & 18.18 & 27  & 51  & 47.06 & 11.76 \\
50\_14  & 45  & 22 & 54  & 59.26 & 16.67 & 24  & 51  & 52.94 & 11.76 \\
50\_15  & 42  & 20 & 52  & 61.54 & 19.23 & 24  & 50  & 52.00 & 16.00 \\
50\_16  & 42  & 22 & 50  & 56.00 & 16.00 & 24  & 48  & 50.00 & 12.50 \\
50\_17  & 39  & 22 & 48  & 54.17 & 18.75 & 22  & 46  & 52.17 & 15.22 \\
50\_18  & 45  & 25 & 52  & 51.92 & 13.46 & 25  & 51  & 50.98 & 11.76 \\
50\_19  & 41  & 21 & 49  & 57.14 & 16.33 & 24  & 47  & 48.94 & 12.77 \\
50\_20  & 45  & 21 & 52  & 59.62 & 13.46 & 24  & 51  & 52.94 & 11.76 \\
50\_21  & 44  & 26 & 52  & 50.00 & 15.38 & 26  & 51  & 49.02 & 13.73 \\
50\_22  & 43  & 21 & 51  & 58.82 & 15.69 & 23  & 49  & 53.06 & 12.24 \\
50\_23  & 40  & 22 & 50  & 56.00 & 20.00 & 24  & 47  & 48.94 & 14.89 \\
50\_24  & 47  & 24 & 58  & 58.62 & 18.97 & 26  & 55  & 52.73 & 14.55 \\
50\_25  & 47  & 23 & 58  & 60.34 & 18.97 & 25  & 55  & 54.55 & 14.55 \\
100\_1  & 53  & 22 & 91  & 75.82 & 41.76 & 35  & 59  & 40.68 & 10.17 \\
100\_2  & 50  & 23 & 90  & 74.44 & 44.44 & 36  & 60  & 40.00 & 16.67 \\
100\_3  & 48  & 21 & 81  & 74.07 & 40.74 & 35  & 54  & 35.19 & 11.11 \\
100\_4  & 49  & 23 & 86  & 73.26 & 43.02 & 35  & 58  & 39.66 & 15.52 \\
100\_5  & 53  & 23 & 85  & 72.94 & 37.65 & 34  & 61  & 44.26 & 13.11 \\
100\_6  & 57  & 24 & 91  & 73.63 & 37.36 & 39  & 65  & 40.00 & 12.31 \\
100\_7  & 51  & 24 & 88  & 72.73 & 42.05 & 34  & 59  & 42.37 & 13.56 \\
100\_8  & 51  & 21 & 87  & 75.86 & 41.38 & 35  & 60  & 41.67 & 15.00 \\
100\_9  & 53  & 24 & 87  & 72.41 & 39.08 & 37  & 63  & 41.27 & 15.87 \\
100\_10 & 52  & 23 & 87  & 73.56 & 40.23 & 36  & 62  & 41.94 & 16.13 \\
100\_11 & 54  & 25 & 89  & 71.91 & 39.33 & 37  & 60  & 38.33 & 10.00 \\
100\_12 & 53  & 23 & 88  & 73.86 & 39.77 & 35  & 61  & 42.62 & 13.11 \\
100\_13 & 53  & 23 & 91  & 74.73 & 41.76 & 36  & 62  & 41.94 & 14.52 \\
100\_14 & 51  & 23 & 85  & 72.94 & 40.00 & 36  & 61  & 40.98 & 16.39 \\
100\_15 & 56  & 23 & 84  & 72.62 & 33.33 & 36  & 60  & 40.00 & 6.67  \\
100\_16 & 49  & 23 & 84  & 72.62 & 41.67 & 34  & 57  & 40.35 & 14.04 \\
100\_17 & 54  & 27 & 85  & 68.24 & 36.47 & 34  & 57  & 40.35 & 5.26  \\
100\_18 & 53  & 23 & 87  & 73.56 & 39.08 & 36  & 59  & 38.98 & 10.17 \\
100\_19 & 50  & 23 & 59  & 61.02 & 15.25 & 36  & 59  & 38.98 & 15.25 \\
100\_20 & 50  & 22 & 90  & 75.56 & 44.44 & 35  & 60  & 41.67 & 16.67 \\
100\_21 & 54  & 26 & 63  & 58.73 & 14.29 & 36  & 63  & 42.86 & 14.29 \\
100\_22 & 50  & 22 & 87  & 74.71 & 42.53 & 34  & 60  & 43.33 & 16.67 \\
100\_23 & 52  & 23 & 86  & 73.26 & 39.53 & 36  & 57  & 36.84 & 8.77  \\
100\_24 & 54  & 26 & 95  & 72.63 & 43.16 & 39  & 64  & 39.06 & 15.63 \\
100\_25 & 53  & 23 & 66  & 65.15 & 19.70 & 36  & 64  & 43.75 & 17.19 \\
200\_1  & 77  & 23 & 126 & 81.75 & 38.89 & 59  & 94  & 37.23 & 18.09 \\
200\_2  & 74  & -  & -   & -     & -     & 59  & 91  & 35.16 & 18.68 \\
200\_3  & 75  & 26 & 123 & 78.86 & 39.02 & 58  & 87  & 33.33 & 13.79 \\
200\_4  & 77  & 25 & 126 & 80.16 & 38.89 & 59  & 91  & 35.16 & 15.38 \\
200\_5  & 74  & -  & -   & -     & -     & 60  & 93  & 35.48 & 20.43 \\
200\_6  & 80  & 25 & 216 & 88.43 & 62.96 & 59  & 96  & 38.54 & 16.67 \\
200\_7  & 71  & -  & -   & -     & -     & 56  & 91  & 38.46 & 21.98 \\
200\_8  & 76  & -  & -   & -     & -     & 60  & 94  & 36.17 & 19.15 \\
200\_9  & 75  & -  & -   & -     & -     & 61  & 93  & 34.41 & 19.35 \\
200\_10 & 77  & 23 & 131 & 82.44 & 41.22 & 59  & 94  & 37.23 & 18.09 \\
200\_11 & 74  & -  & -   & -     & -     & 59  & 93  & 36.56 & 20.43 \\
200\_12 & 80  & 26 & 125 & 79.20 & 36.00 & 58  & 94  & 38.30 & 14.89 \\
200\_13 & 75  & 23 & 131 & 82.44 & 42.75 & 57  & 95  & 40.00 & 21.05 \\
200\_14 & 79  & 24 & 136 & 82.35 & 41.91 & 64  & 97  & 34.02 & 18.56 \\
200\_15 & 77  & -  & -   & -     & -     & 57  & 90  & 36.67 & 14.44 \\
200\_16 & 74  & 25 & 125 & 80.00 & 40.80 & 54  & 89  & 39.33 & 16.85 \\
200\_17 & 74  & 24 & 126 & 80.95 & 41.27 & 55  & 87  & 36.78 & 14.94 \\
200\_18 & 81  & 25 & 131 & 80.92 & 38.17 & 60  & 90  & 33.33 & 10.00 \\
200\_19 & 75  & 23 & 130 & 82.31 & 42.31 & 59  & 88  & 32.95 & 14.77 \\
200\_20 & 75  & -  & -   & -     & -     & 57  & 93  & 38.71 & 19.35 \\
200\_21 & 77  & -  & -   & -     & -     & 60  & 96  & 37.50 & 19.79 \\
200\_22 & 74  & 22 & 244 & 90.98 & 69.67 & 56  & 89  & 37.08 & 16.85 \\
200\_23 & 70  & 23 & 124 & 81.45 & 43.55 & 55  & 88  & 37.50 & 20.45 \\
200\_24                   & 82                       & 26 & 140 & 81.43        & 41.43    & 64 & 100 & 36.00        & 18.00    \\
200\_25 & 75  & 24 & 137 & 82.48 & 45.26 & 58  & 95  & 38.95 & 21.05 \\
400\_1  & 126 & -  & -   & -     & -     & 107 & 156 & 31.41 & 19.23 \\
400\_2  & 124 & -  & -   & -     & -     & 108 & 158 & 31.65 & 21.52 \\
400\_3  & 125 & -  & -   & -     & -     & 26  & 153 & 83.01 & 18.30 \\
400\_4  & 127 & -  & -   & -     & -     & 25  & 157 & 84.08 & 19.11 \\
400\_5  & 121 & -  & -   & -     & -     & 109 & 156 & 30.13 & 22.44 \\
400\_6  & 129 & -  & -   & -     & -     & 107 & 167 & 35.93 & 22.75 \\
400\_7  & 119 & -  & -   & -     & -     & 103 & 157 & 34.39 & 24.20 \\
400\_8  & 127 & -  & -   & -     & -     & 112 & 163 & 31.29 & 22.09 \\
400\_9  & 116 & -  & -   & -     & -     & 111 & 162 & 31.48 & 28.40 \\
400\_10 & 125 & -  & -   & -     & -     & 107 & 160 & 33.13 & 21.88 \\
400\_11 & 118 & -  & -   & -     & -     & 104 & 154 & 32.47 & 23.38 \\
400\_12 & 120 & -  & -   & -     & -     & 109 & 153 & 28.76 & 21.57 \\
400\_13 & 122 & -  & -   & -     & -     & 105 & 162 & 35.19 & 24.69 \\
400\_14 & 132 & -  & -   & -     & -     & 23  & 161 & 85.71 & 18.01 \\
400\_15 & 123 & -  & -   & -     & -     & 102 & 152 & 32.89 & 19.08 \\
400\_16 & 119 & -  & -   & -     & -     & 98  & 150 & 34.67 & 20.67 \\
400\_17 & 121 & -  & -   & -     & -     & 99  & 148 & 33.11 & 18.24 \\
400\_18 & 132 & -  & -   & -     & -     & 112 & 159 & 29.56 & 16.98 \\
400\_19 & 129 & -  & -   & -     & -     & 112 & 154 & 27.27 & 16.23 \\
400\_20 & 118 & -  & -   & -     & -     & 104 & 156 & 33.33 & 24.36 \\
400\_21 & 124 & -  & -   & -     & -     & 108 & 159 & 32.08 & 22.01 \\
400\_22 & 121 & -  & -   & -     & -     & 102 & 153 & 33.33 & 20.92 \\
400\_23 & 112 & -  & -   & -     & -     & 101 & 144 & 29.86 & 22.22 \\
400\_24 & 130 & -  & -   & -     & -     & 114 & 166 & 31.33 & 21.69 \\
400\_25 & 122 & -  & -   & -     & -     & 105 & 160 & 34.38 & 23.75
\end{longtable}
\end{document}